\documentclass[letterpaper, 10 pt, conference]{ieeeconf}  
\usepackage{amsmath} 
\usepackage{amssymb}  
\usepackage[dvipsnames]{xcolor}

\IEEEoverridecommandlockouts                              
\usepackage{color}
\usepackage{multirow}
\usepackage{graphicx}
\usepackage{subfig}
\usepackage{epstopdf}
\usepackage{lipsum}  
\usepackage{mathtools,caption,float,yfonts,cuted}
\usepackage{balance}
\newtheorem{definition}{Definition}
\newtheorem{theorem}{Theorem}
\newtheorem{lemma}{Lemma}
\newtheorem{remark}{Remark}

\newtheorem{assumption}{Assumption}
\newtheorem{example}{Example}
\newtheorem{stassumption}[assumption]{Standing Assumption}

\newtheorem{corollary}{Corollary}
\newcommand{\argmin}{\mathop{\mathrm{argmin}}\limits}

\newcommand{\Id}{ \textrm{Id}}
\newcommand{\bR} { {\mathbb R}}
\newcommand{\bN} { {\mathbb N}}

\newcommand{\fix} { {\mathrm{fix}}}
\newcommand{\zer} { {\mathrm{zer}}}
\newcommand{\diag} { {\mathrm{diag}}}
\newcommand{\prox}{\textrm{prox}}
\newcommand{\proj}{\textrm{proj}}

\newcommand{\ca}[1]{\mathcal{#1}}

\newcommand{\bld}[1]{\boldsymbol{#1}}
\newcommand{\1}{\bld 1}
\newcommand{\0}{\bld 0}
\newcommand{\col }{\mathrm{col}}

\usepackage{enumerate}

\title{\LARGE \bf
Asynchronous and time-varying proximal type dynamics\\ in multi-agent network games 
}

\author{Carlo Cenedese \and Giuseppe Belgioioso \and Yu Kawano \and Sergio Grammatico \and Ming Cao 
\thanks{C. Cenedese and M. Cao are with the Jan C. Wilems Center for Systems and Control, ENTEG, Faculty of Science and Engineering, University of Groningen, The Netherlands	({\texttt{\{c.cenedese, m.cao\}@rug.nl}})
 G.Belgioioso is with the Control System group, TU Eindhoven, 5600 MB Eindhoven, The Netherlands
(\texttt{g.belgioioso@tue.nl}).
Y. Kawano is with the Graduate School of Engineering, Hiroshima University, Japan (\texttt{ykawano@hiroshima-u.ac.jp}).
 S. Grammatico is with the Delft Center for Systems and Control, TU Delft, The Netherlands
(\texttt{{s.grammatico@tudelft.nl}}).
This work was partially supported by the EU Project \lq MatchIT' (82203), NWO under research project OMEGA (613.001.702) and P2P-TALES (647.003.003) and by the ERC under research project COSMOS (802348).}
}

\begin{document}

\maketitle
\thispagestyle{empty}
\pagestyle{empty}

\begin{abstract}
In this paper, we study proximal type dynamics in the context of noncooperative multi-agent network games. These dynamics arise in different applications, since they describe distributed decision making in multi-agent networks, e.g., in opinion dynamics, distributed model fitting and network information fusion, where the goal of each agent is to seek an equilibrium using local information only.
We analyse several conjugations of this class of games, providing convergence results, or designing equilibrium seeking algorithms when the original dynamics fail to converge.
For the games subject only to local constraints we look into both synchronous/asynchronous dynamics and time-varying communication networks. For games subject in addition to coupling constraints, we design an equilibrium seeking algorithm converging to a special class of game equilibria. 
Finally, we validate the theoretical results via numerical simulations on opinion dynamics and distributed model fitting.
\end{abstract}
\section{Introduction}

\subsection{Motivation: Multi-agent decision making over networks}
Multi-agent decision making over networks is currently a vibrant research area in the systems-and-control community, with applications in several relevant domains,  such as smart grids \cite{dorfer:simpson-porco:bullo:16, grammatico:18tcns}, traffic and information networks \cite{jaina:walrand:10}, social networks \cite{Ghaderi2014, etesami:basar:15}, consensus and flocking groups \cite{cao:yu:ren:2013:overview_distributed_multy_agent_coordination} and robotic \cite{cenedese:kawano:grammatico:cao:2018:time_varying_proximal_dynamics} and sensor networks \cite{martinez:bullo:cortes:frazzoli:07}, \cite{simonetto:leus:2014:distributed_ML_sensor_network}.
The main benefit that each decision maker, in short, \textit{agent}, achieves from the use of a distributed computation and communication, is to keep its own data private and to exchange information with selected agents only.  
Essentially, in networked multi-agent systems, the states (or the decisions) of each agent evolve as a result of \textit{local decision making}, e.g. local constrained optimization, and \textit{distributed communication} with some neighboring agents, via a communication network, usually modelled by a graph. Usually, the aim of the agents is to reach a collective equilibrium state, where no one can benefit from changing its state at that  equilibrium.

\subsection{Literature overview: Multi-agent optimization and multi-agent network games}

Multi-agent dynamics over networks embody the natural extension of  distributed optimization and equilibrium seeking problems in network games. In the past decade, this field has drawn attentions from research communities, leading to a wide range of results. 
Some examples of convex optimization constrained problems, subject to homogeneous constraint sets, can be found in \cite{nedic:ozdaglar:parrillo:10}, where a uniformly bounded subgradients and a complete communication graphs with uniform weights are considered; while in \cite{lee:nedic:13} the cost functions are assumed to be differentiable with Lipschitz continuous and uniformly bounded gradients.

Solutions for nooncoperative games over networks subject to convex compact local constraints have been developed also in \cite{parise:gentile:grammatico:lygeros:15} using strongly convex quadratic cost functions and time-invariant communication graphs; in \cite{koshal:nedic:shanbhag:16} \cite{salehisadaghiani:pavel:16}, using differentiable cost functions with Lipschitz continuous gradients, strictly convex cost functions, and undirected, possibly time-varying, communication graphs; and in \cite{cenedese:kawano:grammatico:cao:2018:time_varying_proximal_dynamics}
 where the communication is ruled by  a, possibly time-varying, digraph and the cost functions are assumed to be convex and lower semi-continuous. 
Some recent works developed algorithms to solve games over networks subject to asynchronous updates of the agents: in \cite{yi:pavel:19:asynch_distr_GNE_seeking} and \cite{cenedese:belg:gra:cao:2019:Asynch_ECC} the game is subject to affine coupling constraints, and the cost functions are differentiables, while the communication graph is assumed to be  undirected.

To the best of the our knowledge,  the only work that extensively focuses on multi-agent network games is \cite{grammatico:18tcns}, which considered general local convex costs and quadratic proximal terms, time-invariant and time-varying communication graphs, subject to technical restrictions. 

This paper aims at presenting and solving  new interesting problems, not yet considered in the context of multi-agent network games and generalizing available results as highlighted in the next section.


\subsection{Contribution of the paper}

Within the litterature on multi-aegnt network games, the most relevant works for our purposes are \cite{parise:gentile:grammatico:lygeros:15,grammatico:18tcns} and \cite{cenedese:kawano:grammatico:cao:2018:time_varying_proximal_dynamics}, the last being the starting point of this paper. Next, we highlight the main contribution of our work with respect to the aforementioned literature and these works in particular:
\begin{itemize}
\item We prove that a row stochastic adjacency matrix with self-loops describing a strongly connected graph is an average operator. Furthermore, we show this holds in a Hilbert space weighted by the diagonal matrix, with entries being the elements of the left Perron Frobenius eigenvector of that matrix. This is a general result that we exploit for our proximal type dynamics.
\item We establish the convergence for synchronous, asynchronous and time-varying dynamics in multi-agent network games, under the weak assumption of a communication network described by a row stochastich adjacency matrix. This weak assumption leads to several technical challenges that in \cite{grammatico:18tcns} are not present, since the matrix is always assumed to be doubly stochastic, and in \cite{cenedese:kawano:grammatico:cao:2018:time_varying_proximal_dynamics} are solved by modifying the original dynamics. 
We emphasize that in \cite{grammatico:18tcns, cenedese:kawano:grammatico:cao:2018:time_varying_proximal_dynamics}, the asynchronous dynamics are not considered.
\item We develop a semi-decentralized algorithm seeking normalized generalized network equilibrium, called Prox-GNWE, converging to a solution of the game subject to affine coupling constraints. Compared to the one developed in \cite[Eq.~28--31]{grammatico:18tcns}, Prox-GNWE  requires a lower number of communications between agents, and it can also be applied to games with a wider class of communication networks.
\end{itemize} 

%

\section{NOTATION}
\subsection{Basic notation}
The set of real, positive, and non-negative numbers are denoted by $\mathbb{R}$, $\mathbb{R}_{>0}$ and $\mathbb{R}_{\geq 0}$, respectively; $\overline{\mathbb{R}}:=\mathbb{R}\cup \{\infty\}$. The set of natural numbers is denoted by $\mathbb{N}$. For a square matrix $A \in \bR^{n\times n}$, its transpose is denoted by $A^\top$ and $[A]_{i}$ denotes the $i$-th row of the matrix, and  $[A]_{ij}$ the element in the $i$-th row and $j$-th column. $A\succ 0 $ ($A\succeq 0 $)  stands for a symmetric and positive definite (semidefinite) matrix,  while $>$ ($\geq$) describes an element wise inequality. $A\otimes B$ is the Kronecker product of the matrices $A$ and $B$.   The identity matrix is denoted by~$I_n\in\bR^{n\times n}$. $\0$ ($\1$) represents the vector/matrix with only $0$ ($1$) elements. For $x_1,\dots,x_N\in\mathbb{R}^n$ and $\ca N=\{1,\dots,N \}$, the collective vector is denoted  by $\boldsymbol{x}:=\mathrm{col}((x_i )_{i\in\ca N})=[x_1^\top,\dots ,x_N^\top ]^\top$ and $\bld x_{-i}:=\col(( x_j )_{j\in\ca N\setminus \{i\}})=[x_1^\top,\dots,x_{i-1}^\top,x_{i+1}^\top,\dots,x_{N}^\top]^\top$. Given the  $N$ matrices $A_1,\dots,A_N$, $\diag(A_1,\dots,A_N)$ denotes a block-diagonal matrix with the matrices $A_1,\dots,A_N$ on the diagonal. The null space of a matrix $A$ is denoted by $\mathrm{ker}(A)$.
The Cartesian product of the sets $\Omega_1, \dots, \Omega_N$ is described by $\prod^N_{i=1} \Omega_i$.  Given two vectors $x,y \in \bR^n$ and a symmetric and positive definite matrix $ Q \succ 0$, the weighted inner product and norm are denoted by $\langle\,  x \, | \, y \,\rangle_{Q}$ and $\lVert x \rVert_{Q}$, respectively; the $Q-$induced matrix norm is denoted by $\lVert A\rVert_{Q}$. A real $n$ dimensional Hilbert space obtained by endowing $\mathcal H=(\bR^n,\lVert \, \cdot \, \rVert)$  with the product $\langle\,  x \, | \, y \,\rangle_{Q}$ is denoted by $\mathcal{H}_Q$.

\subsection{Operator-theoretic notations and definitions}
The identity operator is defined by $\Id(\cdot)$. The indicator function $\iota_\mathcal{C}:\bR^n\rightarrow[0,+\infty]$ of $\mathcal{C}\subseteq \bR^n$ is defined as $\iota_\mathcal{C}(x)=0$ if $x\in\mathcal{C}$; $+\infty$  otherwise.
The set valued mapping $N_{\ca C}:\bR^n\rightrightarrows \bR^n$ stands for the normal cone to the set $\mathcal{C}\subseteq \bR^n$, that is $N_{\ca C}(x)= \{ u\in\bR^n \,|\, \mathrm{sup}\langle \ca C-x,u \rangle\leq 0\}$ if $x \in \ca C$ and  $\varnothing$ otherwise. The graph of a set valued mapping $\ca A:\ca X\rightrightarrows \ca Y$ is $\mathrm{gra}(\ca A):= \{ (x,u)\in \ca X\times \ca Y\, |\, u\in\ca A (x)  \}$. For a function $\phi:\bR^n\rightarrow\overline{\mathbb{R}}$, define $\mathrm{dom}(\phi):=\{x\in\bR^n|f(x)<+\infty\}$ and its subdifferential set-valued mapping, $\partial \phi:\mathrm{dom}(\phi)\rightrightarrows\bR^n$, $\partial \phi(x):=\{ u\in \bR^n | \: \langle y-x|u\rangle+\phi(x)\leq \phi(y)\, , \: \forall y\in\mathrm{dom}(\phi)\}$.  The projection operator over a closed set $S\subseteq \bR^n$ is $\textrm{proj}_S(x):\bR^n\rightarrow S$ and it is defined as $\textrm{proj}_S(x):=\mathrm{argmin}_{y\in S}\lVert y - x \rVert^2$. The proximal operator $\mathrm{prox}_f(x):\bR^n\rightarrow\mathrm{dom}(f)$ is defined by $\mathrm{prox}_f(x):=\mathrm{argmin}_{y\in\bR^n}f(y)+\textstyle\frac{1}{2}\lVert x-y\rVert^2$.
A set valued mapping $\ca F:\bR^n\rightrightarrows \bR^n$ is $\ell$-Lipschitz continuous with $\ell>0$, if $\lVert u-v \rVert \leq \ell \lVert x-y \rVert$ for all $(x,u)\, ,\,(y,v)\in\mathrm{gra}(\ca F)$; $\ca F$ is (strictly) monotone if for all $(x,u),(y,v)\in\mathrm{gra}(\ca F)$ $\langle u-v,x-y\rangle \geq (>)0$ holds, and  maximally monotone if it there is no monotone operator with a graph that strictly contains $\mathrm{gra}(\ca F)$; $\ca F$ is $\alpha$-strongly monotone if for all $(x,u),(y,v)\in\mathrm{gra}(\ca F)$ it holds $\langle x-y, u-v\rangle \geq \alpha \lVert x-y \rVert^2$. $\mathrm{J}_{\ca F} :=(\Id+\ca F)^{-1}$ denotes the resolvent mapping of $\ca F$. $\textrm{fix}(\ca A):=\{x \in \mathbb R^n | \, x \in \ca F(x)  \}$ and $\textrm{zer}(\ca F):=\{x \in \mathbb R^n | \,0 \in \ca F(x)  \}$ denote the set of fixed points and zeros of $\ca F$, respectively.
The operator $\ca A:\bR^n\rightarrow\bR^n$  is $\eta$-averaged  ($\eta$-AVG) in $\ca H_Q$, with $\eta\in(0,1)$, if $\lVert \ca A(x)-\ca A(y) \rVert^2_Q \leq \lVert x-y\rVert^2_Q-\frac{1-\eta}{\eta}\lVert (\Id-\ca A)(x)-(\Id-\ca A)(y)  \rVert^2_Q$, for all $x,y\in\bR^n$; $\ca A$ is nonexpansive (NE) if $1$-AVG; $\ca A$ is firmly nonexpansive (FNE) if $\frac{1}{2}$-AVG; $\ca A$ is $\beta$-cocoercive if $\beta\ca A$ is $\frac{1}{2}$-AVG (i.e., FNE).
%


\section{Mathematical setup and problem formulation} 
\label{sec:problem_formulation}

We consider a set of $N$ agents (or players), where the state (or strategy) of each agent $i\in\mathcal{N}:= \{1,\dots,N\}$ is denoted by $x_i\in\Omega_i\subset \mathbb{R}^n$. The set $\Omega_i$ represents all the feasible states of agent $i$, hence it is used to model the local constraints of each agent. Throughout the paper, we assume compactness and convexity of the local constraint set~$\Omega_i$.    

\smallskip
\begin{stassumption}[Convexity]
\label{ass:local_constrain}
For each $i\in\mathcal{N}$, the set $\Omega_i\subset \mathbb{R}^n$ is non-empty, compact and convex.
\hfill\QEDopen
\end{stassumption}
\smallskip

We consider rational (or myopic) agents, namely, each agent $i$ aims at minimizing a local cost function $g_i$, that we assume convex and with the following structure.
 
\smallskip
\begin{stassumption}[Proximal cost functions]
\label{ass:cost_function}
For each $i\in\mathcal{N}$, the strictly-convex function $g_i:\mathbb{R}^n\times \mathbb{R}^n\rightarrow\overline{\mathbb{R}}$ is defined by 
\begin{equation}
g_i(x_i,z) :=  f_{i}(x_i) + \iota_{\Omega_i}(x_i)+\textstyle{\frac{1}{2}}\lVert x_i-z\rVert^2,
\label{eq:def_costFunction_Gi}
\end{equation}
where $\bar f_i:= f_{i}+\iota_{\Omega_i}:\mathbb{R}^n\rightarrow\overline{\mathbb{R}}$ is a lower semi-continuous and convex function. 
\hfill\QEDopen
\end{stassumption}
\smallskip

We emphasize that Standing Assumption~\ref{ass:cost_function} requires neither the differentiability of the local cost function, nor the Lipschitz continuity and boundedness of its gradient.
In \eqref{eq:def_costFunction_Gi}, the function $\bar f_{i}$  is local to agent $i$ and models the local objective that the player would pursue if no coupling between agents is present. The quadratic term $\textstyle{\frac{1}{2}}\lVert x_i-z\rVert^2$ penalizes the distance between the state of agent $i$ and a given $z$. 
This term is referred to the literature as \textit{regularization } (see \cite[Ch.~27]{bauschke:combettes}), since it leads to a strictly convex $g_i$, even though the $f_i$ is only lower semi-continuous, see \cite[Th.~27.23]{bauschke:combettes}. 

\smallskip
We assume that the agents can communicate through a network structure, described by a weighted digraph. Let us represent the communication links between the agents by a weighted adjacency matrix $A\in\mathbb{R}^{N\times N}$ defined as $[A]_{ij}:=a_{i,j}$. For all $i,j \in \ca N$, $a_{i,j}\in[0,1]$ denotes the weight that agent $i$ assigns to the state of agent $j$. If $a_{i,j}=0$, then the state of agent $i$ is independent from that of agent $j$. The set of agents with whom agent $i$ communicates is denoted by $\ca N_i$.
The following assumption formalizes the communication network via a digraph and the associated adjacency matrix.

\smallskip
\begin{stassumption}[row stochasticity and self-loops]
\label{ass:row_stoch}
The communication graph is strongly connected. The matrix $A=[a_{i,j}]$ is row stochastic, i.e., $a_{i,j} \geq 0$ for all $i,j \in \ca N$, 
and $\sum_{j=1}^N a_{i,j}=1$, for all $i \in \ca N$. Moreover, $A$ has strictly-positive diagonal elements, i.e., $\min_{i\in\ca N}a_{i,i}=:\underline{a}>0$.~\hfill\QEDopen
\end{stassumption}
\smallskip

In our setup, the variable $z$ in \eqref{eq:def_costFunction_Gi} represents the average state among the neighbors of agent $i$, weighted through the adjacency matrix, i.e., $z\coloneqq \textstyle\sum_{j=1}^N a_{i,j}x_j$. Therefore the cost function of agent $i$ is $g_i \big( y,\textstyle\sum_{j= 1}^N a_{i,j}x_j \big)$. Note that a \textit{coupling} between the agents emerges in the local cost function, due to the dependence  on the other agents strategy.   
The problem just described can be naturally formalized as a noncooperative network game between the $N$ players, i.e.,
\begin{align}
\label{eq:unc_game_CF}
\forall i \in \ca N: \;
\begin{cases}
\textstyle
\argmin_{y \in \mathbb R^n}& \textstyle
 f_i(y)  + \frac{1}{2} \left\| y - \sum_{j=1}^N a_{i,j}x_j \right\|^2\\
\text{ s.t. } &  y \in \Omega_i.
\end{cases}
\end{align}

We consider rational agents that, at each time instant $k \in \mathbb N$, given the states of the neighbors, update their states/strategies according to the following myopic dynamics:
\begin{equation}
x_i(k+1) =  \argmin_{y\in\Omega_i} \; g_i \left(y,\textstyle\sum_{j=1}^N a_{i,j}x_j(k) \right) \:.
\label{eq:unc_game}
\end{equation}

These dynamics are relatively simple, yet arise in diverse  research areas. Next, we recall some papers in the literature where the dynamics studied are a special case of \eqref{eq:unc_game}.  

\begin{enumerate}
\item {\it Opinion dynamics:} In \cite{parsegov_friedkin:2017:novel_models_opinion_dynamics}, the authors study the Friedkin-Johnsen model, that is an extension of the DeGroot's model~\cite{degroot:1974:reaching_consensus}. The update rule in \cite[Eq.~1]{parsegov_friedkin:2017:novel_models_opinion_dynamics} is effectively the best response of a game with cost functions equal to
\begin{equation}
\label{eq:FJ cost fun}
g_i(x_i,z)\coloneqq \textstyle{\frac{1-\mu_i}{\mu_i}} \lVert x_i-x_i(0) \rVert^2+ \iota_{[0,1]^n}(x_i) + \lVert  x_i- z \rVert^2
\end{equation}
 where $\mu_i \in [0,1]$ represents the stubbornness of the player.
Thus, \cite[Eq.~1]{parsegov_friedkin:2017:novel_models_opinion_dynamics} is a special case of~\eqref{eq:unc_game}.

\smallskip
\item {\it Distributed model fitting:} One of the most common tasks in machine learning is model fitting and in fact several algorithms are proposed in literature, e.g. \cite{zou:2005:elasticnet,tibshirani:1996:lasso}. The idea is to identify the parameters $x$ of a linear model $Ax=b$, where $A$ and $b$ are obtained via experimental data.
If there is a large number of data, i.e., $A$ is a tall matrix, then the  distributed counterpart of these algorithms are presented in \cite[Sec.~8.2]{boyd:parikh:chu:2011:statistical_learning_ADMM}. In particular,  \cite[Eq.~8.3]{boyd:parikh:chu:2011:statistical_learning_ADMM} can be rewritten as a constrained version of \eqref{eq:unc_game}. 
The cost function is defined by
$$g(x_i,z)\coloneqq \ell_i (A_i x_i-b_i) + r(z)\,,$$
where $A_i$ and $b_i$ represent the $i$-th block of available data,and $x_i$ is  the local estimation of the parameters of the model. Here, $\ell_i$ is a loss function and $r$ is the regularization function $r(z)\coloneqq \frac{1}{2}\lVert x_i - z \rVert $. Finally, the arising game is subject to the constraint that at the equilibrium $x_i=x_j$ for all $i\in\ca N$. In Section~\ref{sec:distr_lasso_algorithm}, we describe in detailed a possible implementation of this type of algorithms.

\smallskip


\smallskip
\item {\it Constrained consensus:} if the cost function of the game in \eqref{eq:unc_game} is chosen with $f_{i}=0$ for all $i\in\ca N$, then we retrive the projected consensus algorithm studied in \cite[Eq.~3]{nedic:ozdaglar:parrillo:10} to solve the problem of constrained consensus. To achieve the convergence to the consensus,  it is  required the additional assumption that $\mathrm{int}(\cap_{i\in\ca N} \Omega_i )\not = \varnothing$, see \cite[Ass.~1]{nedic:ozdaglar:parrillo:10}.
\end{enumerate}

\medskip
Next, we introduce the concept of equilibrium of interest in this paper.
Informally, we say that a collective vector $\bar{\bld x}\in \boldsymbol \Omega :=\prod_{i=1}^N \Omega_i$ is an equilibrium of the dynamics~\eqref{eq:unc_game} of the game in \eqref{eq:unc_game_CF}, if no agent $i$ can decrease its local cost function by unilaterally changing its strategy with another feasible one. We formalize this concept in the following definition.


\smallskip
\begin{definition}[{Network equilibrium \cite[Def.~1]{grammatico:18tcns}}]\label{def:NetworkEquilibrium}
The collective vector $\bar{\boldsymbol x}= \col( (\bar{x}_i)_{i\in\ca N} )$ is a \textit{network equilibrium} (NWE) if, for every $i\in\ca N$,
\begin{equation}
\overline{x}_i =   \argmin_{y\in\bR^n} \; g_i \left(y,\textstyle\sum_{j=1}^N a_{i,j}\overline{x}_j \right)\:.
\label{eq:NetworkEquilibrium}
\end{equation}
\hfill\QEDopen
\end{definition}
\smallskip
  
We note that if there are no self loops in the adjacency matrix, i.e., $a_{i,i}=0$ for all $i \in\ca{N}$, then an NWE corresponds to a Nash equilibrium \cite[Remark 1]{grammatico:18tcns}. 

In the next section, we study the convergence of the dynamics in \eqref{eq:unc_game} to an NWE for a population of myopic agents. In particular, when the dynamics do not converge, we propose alternative ones to recover convergence.

\section{Proximal dynamics}
\label{sec:static_dynamics}

In this section, we study three different types of unconstrained proximal dynamics, namely synchronous, asynchronous and time-varying. While for the former two, we can study and prove convergence to an NWE, the last does not ensure convergence. Thus, we propose a modified version of the dynamics with the same equilibria of the original game.  

\subsection{Unconstrained dynamics}
\label{sec:unconstr_dyn}

As a first step, we rephrase the dynamics in a more compact form by means of the proximity operator. In fact, the dynamics in \eqref{eq:unc_game}, with cost function as in \eqref{eq:def_costFunction_Gi}, are equivalent to the following proximal dynamics: for all $i\in\ca N$,
\begin{equation}\label{eq:Banach_dynamics_i}
x_i(k+1)  = \textrm{prox}_{ \bar f_i}\big(\textstyle\sum_{j=1}^N a_{i,j} \, x_j(k) \big),\:   \ \forall k \in \mathbb{N}.
\end{equation}
In compact form, they read as 
\begin{equation}\label{eq:group_dynamics}
\boldsymbol x(k+1) = \boldsymbol{\prox_{f}}( \boldsymbol{A}\,\boldsymbol x(k) ) \,,
\end{equation}
where the matrix $\boldsymbol{A} := A \otimes I_n$ represents the interactions among agents, and the mapping 
$\boldsymbol{\textrm{prox}_{\boldsymbol{f}}}$ is a block-diagonal proximal operator, i.e., 
\begin{equation} \label{eq:group_proximity_operator}
\boldsymbol{\prox_{f}} \left(
\begin{bmatrix}
z_1 \\
\vdots \\
z_N
\end{bmatrix}
\right) := 
\begin{bmatrix}
\textrm{prox}_{\bar f_1}(z_1) \\
\vdots\\
\textrm{prox}_{\bar f_N}(z_N)
\end{bmatrix}.
\end{equation}


\smallskip
\begin{remark}\label{rem:NWE_as_fix}
The definition of NWE can be equivalently recast as a fixed point of the operator \eqref{eq:group_dynamics}. In fact, a collective vector $\overline{\boldsymbol{x}}$ is an NWE if and only if $\overline{\boldsymbol{x}}\in\textrm{fix}(\boldsymbol{\textrm{prox}_{\boldsymbol{f}}}\circ \boldsymbol{A})$. Under Assumptions~\ref{ass:local_constrain}~and~\ref{ass:cost_function}, $\textrm{fix}\left(\boldsymbol{\mathrm{prox}_{\boldsymbol{f}}}\circ \boldsymbol{A} \right)$ is non-empty \cite[Th.~4.1.5]{smart1980fixedpoint_theory}, i.e., there always exists an NWE of the game, thus the convergence problem is well posed. \hfill\QEDopen
\end{remark}
\smallskip


The following lemma represents the cornerstone used to prove the convergence of the dynamics in  \eqref{eq:group_dynamics} to an NWE. In particular, it shows that a row stochastic matrix $A$ is an AVG operator in the Hilbert space weighted by a diagonal matrix $Q$, whose diagonal entries are the elements of the left Perron-Frobenius (PF) eigenvector of $A$, i.e., a vector $\bar q\in\bR^{N}_{>0}$ s.t. $\bar q^\top A= \bar q^\top $. In the remainder of the paper, we always consider the  normalized PF eigenvector, i.e., $q = \bar q/\lVert \bar q \rVert$.
\smallskip


\begin{lemma}[Averageness and left PF eigenvector]\label{lem:proxA_is_AVG_in_H_Q}
Let Assumption 3 hold true, i.e, $A$ be row stochastic, $\underline{a}>0$ be its smallest diagonal element and $q = \col(q_1,\ldots,q_N)$ denotes its left PF eigenvector. Then, the following hold:
\begin{enumerate}[(i)]

\item $A$ is $\eta$-AVG in $\ca H_{ Q}$, with  $Q\coloneqq\diag (q_1, \ldots,q_N)$ and $\eta \in (0,1-\underline{a})$;

\item The operator $ \boldsymbol{\mathrm{prox}_{ \boldsymbol{f}}}\circ \boldsymbol{A} $ is  $\frac{1}{2-\eta}$--AVG in $\ca H_{ Q}$ 

\end{enumerate}

If $A$ is doubly-stochastic, the  statements  hold with $Q=I$.

  \quad\hfill $\square$
\end{lemma}

\smallskip 
\begin{proof}
See Appendix~\ref{app:proof_th_1}. 
\end{proof}
\smallskip


Now, we are ready to present the first result, namely, the global convergence of the proximal dynamics in \eqref{eq:group_dynamics} to an NWE of the game in \eqref{eq:unc_game_CF}. 

\smallskip
\begin{theorem}[Convergence of proximal dynamics]
\label{th:convergence_synch_unc}
For any $\boldsymbol x(0) \in \bold \Omega $, the sequence $\left( \boldsymbol x(k) \right)_{k\in\bN}$ generated by the proximal dynamics in \eqref{eq:group_dynamics} converges to an NWE of \eqref{eq:unc_game_CF}.
\hfill\QEDopen
\end{theorem}
\begin{proof}
See Appendix~\ref{app:proof_th_1}.
\end{proof}

\smallskip
\begin{remark} \label{rem:diagonal}
Theorem \ref{th:convergence_synch_unc} extends \cite[Th. 1]{grammatico:18tcns}, where the matrix $A$ is assumed doubly-stochastic \cite[Ass.~1~Prop.~2]{grammatico:18tcns}.  In this case, $\1_N$ is the left PF eigenvector of $A$ and the matrix $Q$ can be set as the identity matrix, see Lemma~\ref{lem:proxA_is_AVG_in_H_Q}.
 \hfill\QEDopen

\end{remark}
\smallskip 
   
%
\begin{remark}
In \cite[Sec.~VIII-A.]{grammatico:18tcns} the authors study, via simulations, an application of the game in \eqref{eq:unc_game_CF} to opinion dynamics. In particular, they conjecture the convergence of the dynamics in the case of a row stochastic weighted adjacency matrix.
Theorem~\ref{th:convergence_synch_unc} theoretically supports the  convergence of this class of dynamics. \hfill\QEDopen 
\end{remark}
\smallskip 

The presence of self-loops in the communication (Assumption~\ref{ass:row_stoch}) is critical for the convergence of the dynamics in \eqref{eq:group_dynamics}. Next, we present a simple example of a two player game, in which the dynamics fail to converge due to the lack of self-loops.
\smallskip

\begin{example} \label{ex:anti_coord_matrix}
Consider the two player game, in which the state of each agent is scalar and $\bar f_i\coloneqq \iota_{\Omega}$, for $i=1,2$. The set $\Omega$ is an arbitrarily big compact subset of $\bR$, in which the dynamics of the agents are invariant. The communication network is described by the doubly-stochastic adjacency matrix $A=\left[\begin{smallmatrix}
0 & 1\\
1 & 0 \end{smallmatrix}\right]$, defining a strongly connected graph without self-loops. In this example, the dynamics in \eqref{eq:group_dynamics} reduce to $\bld x(k+1) = A\bld x(k)$. Hence, convergence does not take place for all $\bld x(0)\in\bld\Omega$.\hfill\QEDopen 
\end{example}
\smallskip


If $\underline a = 0$, i.e., the self-loop requirement  in Standing Assumption~\ref{ass:row_stoch} is not satisfied, then the convergence can be restored by relaxing the dynamics in \eqref{eq:unc_game} via the so-called Krasnoselskii iteration \cite[Sec.~5.2]{Bauschke2010:ConvexOptimization},
\begin{equation}\label{eq:krasno_group_dynamics}
\boldsymbol x(k+1) = (1-\alpha)\boldsymbol x(k)+\alpha\, \boldsymbol{\textrm{prox}}_{\boldsymbol{f}}( \boldsymbol{A}\,\boldsymbol x(k) ) \,,
\end{equation}
where $\alpha\in(0,1)$.
These new dynamics share the same fixed points of \eqref{eq:group_dynamics}, that correspond to NWE of the original game in \eqref{eq:unc_game_CF}.
\smallskip
\begin{corollary}
\label{cor:krasno_conv}
For any $\boldsymbol x(0) \in \bold \Omega$ and for $\min_{i\in\ca N} a_{i,i}\geq 0$, the sequence $\left( \boldsymbol x(k) \right)_{k\in\bN}$ generated by the dynamics in \eqref{eq:krasno_group_dynamics} converges to an NWE of \eqref{eq:unc_game}.
\hfill\QEDopen
\end{corollary}
\smallskip
\begin{proof}
See Appendix \ref{app:proof_th_1}.
\end{proof}

 
\subsection{Asynchronous unconstrained dynamics}
\label{sec:asynchrnous_unconstr}
The dynamics introduced in \eqref{eq:group_dynamics} assume that all the agents update their strategy synchronously. Here, we study the more realistic case in which the agents behave asynchronously, namely, perform their local updates at different time instants. 
Thus, at each time instant $k$, only one agent $i_k\in\ca N$ updates its state according to \eqref{eq:Banach_dynamics_i}, while the others do not, i.e.,
\begin{equation}
\label{eq:dynamics_asynch_i}
x_i(k+1)=\begin{cases}
\textrm{prox}_{\bar f_i}\big(\textstyle\sum_{j=1}^N a_{i,j} x_j(k) \big), & \text{if } i = i_k, \\
x_i(k), & \text{otherwise}.
\end{cases}
\end{equation} 

Next, we derive a compact form for the dynamics above. 
Define  $H_i$  as the matrix of all zeros except for $[H_i]_{ii}=1$, and also $\bld H_i := H_i\otimes I_n $. Then, we define the set $\bld H:=\{\bld H_i\}_{i\in\ca N}$ as a collection of these $N$ matrices.
  At each time instant $k\in\bN$, the choice of an agent which performs the update is modelled via an i.i.d. random variable $\zeta_k$, taking values in $\bld H$. If $\zeta_k=\bld H_i$, it means that agent $i$ updates at time $k$, while the others do not change their strategies. Given a discrete probability distribution $(p_1,\dots,p_N)$, we define $\mathbb{P}[\zeta_k= \bld H_i] = p_i$, for all $i\in\ca N$. With this notation in mind, the dynamics in \eqref{eq:group_dynamics} are modified to model asynchronous updates, 
\begin{equation}\label{eq:dynamics_ARock_unc_mid}
\boldsymbol x(k+1) = \boldsymbol x(k) +  \zeta_k  \big( \boldsymbol{\textrm{prox}}_{\boldsymbol{f}}( \boldsymbol{A}\,\boldsymbol x(k) ) - \boldsymbol x(k)\big) \,.
\end{equation}
We remark that \eqref{eq:dynamics_ARock_unc_mid} represent the natural asynchronous counterpart of the ones in \eqref{eq:group_dynamics}. In fact, the update above is equivalent to the one in \eqref{eq:Banach_dynamics_i} for the active agent at time $k\in \bN$. 


Hereafter, we assume that each agent $i\in\ca N$ has a public and private memory. If the player is not performing an update, the strategies stored in the two memories coincide. During an update, instead, the public memory stores the strategy of the agent before the update has started, while in the private one there is the value that is modified during the computations. 
When the update is completed, the  value in the public memory is overwritten by the private one. This assumption ensures that all the reads of the public memory of each agent $i$, performed by each neighbor $j\in\ca N_i$, are always consistent, see \cite[Sec.~1.2]{Peng-Yan-Xu:2016:ARock} for technical details.

We consider the case in which the computation time for the update is not negligible, therefore the strategies that agent $i$ reads from each neighbor $j\in\ca N_i$ can be outdated of $\varphi_j(k)\in\bN$ time instants. 
 The maximum delay is assumed uniformly upper bounded.
\smallskip
\begin{assumption}[Bounded maximum delay]
\label{ass:bounded_delays}
The delays are uniformly upper bounded, i.e., $\sup_{k\in\bN}\max_{i\in \ca N} \varphi_i(k) \leq \overline{\varphi}<\infty$, for some $\overline{\varphi}\in\bN$. \hfill\QEDopen
\end{assumption} 
\smallskip

The dynamics describing the asynchronous update with delays can be then cast in a compact form as
\begin{equation}\label{eq:dynamics_ARock_unc}
\boldsymbol x(k+1) = \boldsymbol x(k) + \zeta_k  \big( \boldsymbol{\textrm{prox}}_{\boldsymbol{f}}( \boldsymbol{A}\,\hat{\boldsymbol x}(k) ) - \hat{\boldsymbol x}(k)\big) \,,
\end{equation}      
where $\hat{\bld x} = \col(\hat{x}_1,\ldots,\hat{x}_N)$ is the vector of possibly delayed strategies. Note that each agent $i$ has always access to the updated value of its strategy, i.e., $\hat x_i = x_i$, for all $i\in\ca N$. 
We stress that the dynamics in \eqref{eq:dynamics_ARock_unc} coincides with \eqref{eq:dynamics_ARock_unc_mid} when no delay is present, i.e., if $\overline{\varphi}=0$. 

The following theorem claims the convergence (in probability) of \eqref{eq:dynamics_ARock_unc} to an NWE of the game in \eqref{eq:unc_game_CF}, when the maximum delay $\overline{\varphi}$ is small enough.
\smallskip
\begin{theorem}[Convergence of asynchronous dynamics]\label{th:convergence_asynch_psi1}
Let Assumption~\ref{ass:bounded_delays} hold true, $p_{\min} := \min_{i\in\ca N} p_i$ and
\begin{equation}
\label{eq:max_delay_bound}
\textstyle{\overline{\varphi} < \frac{N\sqrt{p_{\min}}}{2(1-\underline{a})}-\frac{1}{2\sqrt{p_{\min}}}}\: .
\end{equation}
Then, for any $\boldsymbol x(0) \in \bold \Omega $, the sequence $( \bld x (k) )_{k\in\bN}$ generated by \eqref{eq:dynamics_ARock_unc} converges almost surely to some $\bar{ \boldsymbol x}\in\fix(\bld{\textrm{prox}_f}\circ \bld A)$, namely, an NWE of the game in \eqref{eq:unc_game_CF}.
\hfill\QEDopen   
\end{theorem}

\smallskip 
\begin{proof}
See Appendix~\ref{app:proof_th_asynch_un}.
\end{proof}
\smallskip
%
If the maximum delay does not satisfy~\eqref{eq:max_delay_bound}, then the convergence of the dynamics in \eqref{eq:dynamics_ARock_unc} is not guaranteed.
In this case, convergence can be restored by introducing a time-varying scaling factor $\psi_k$ in the dynamics: 
\begin{equation}\label{eq:dynamics_ARock_unc_psi}
\boldsymbol x(k+1) = \boldsymbol x(k) + \psi_k \zeta_k  \big( \boldsymbol{\textrm{prox}}_{\boldsymbol{f}}( \boldsymbol{A}\,\hat{\boldsymbol x}(k) ) - \hat{\boldsymbol x}(k)\big) \,.
\end{equation} 
The introduction of this scaling factor, always smaller than $1$, leads to a slower convergence of \eqref{eq:dynamics_ARock_unc_psi} with respect to \eqref{eq:dynamics_ARock_unc}, since it implies a smaller step size in update. 
The next theorem proves that, the modified dynamics converges if the scaling factor is chosen small enough.
 
\smallskip
\begin{theorem}
\label{th:convergence_asynch}
Let Assumption~\ref{ass:bounded_delays} hold true and set
\begin{align*} \textstyle
0<\psi_k<\frac{Np_{\min}}{(2\overline \varphi \sqrt{p_{\min}} +1) (1-\underline{a})}, \quad \forall k \in \mathbb N.
\end{align*}
Then, for any $\boldsymbol x(0) \in \bld \Omega $, the sequence $( \bld x (k) )_{k\in\bN}$ generated by \eqref{eq:dynamics_ARock_unc_psi} converges almost surely  to some $\bar{\bld x} \in \fix(\bld{\textrm{prox}_f}\circ \bld A)$, namely, an NWE of the game in  \eqref{eq:unc_game_CF}.
\hfill\QEDopen   
\end{theorem}

\smallskip 
\begin{proof}
See Appendix~\ref{app:proof_th_asynch_un}.
\end{proof}
\smallskip

\subsection{Time-varying unconstrained dynamics}
\label{sec:time_var_unconstr}

A challenging problem related to the dynamics in \eqref{eq:group_dynamics} is studying its convergence when the communication network is time-varying (i.e., the associated adjacency matrix $A$ is time dependent). In this case, the update rules in \eqref{eq:group_dynamics} become
\begin{equation}\label{eq:dynamics_tv_A}
\boldsymbol x(k+1) = \boldsymbol{\textrm{prox}}_{\boldsymbol{f}}( \boldsymbol{A}(k)\,\boldsymbol x(k) ) \,,
\end{equation}
where $\boldsymbol{A}(k) = A(k) \otimes I_n$ and $A(k)$ is the adjacency matrix at time instant $k$.

Next, we assume \textit{persistent stochasticity} of the sequence $(\bld A(k))_{k \in \mathbb N}$, that can be seen as the time-varying counterpart of Standing Assumption~\ref{ass:row_stoch}. Similar assumptions can be found in several other works over time-varying graphs, e.g., \cite[Ass.~1]{blondel:convergence_multiagent_coord}, \cite[Ass.~4,5]{grammatico:18tcns}, \cite[Ass.~2,3]{nedic:ozdaglar:parrillo:10}.
\smallskip
\begin{assumption}[Persistent row stochasticity and self-loops]\label{ass:persistent_matrixA}
For all $k\in\bN$, the adjacency matrix $\bld A(k)$ is row stochastic and describes a strongly connected graph. Furthermore, there exists $\overline k\in\bN$ such that, for all $k>\overline{k}$, the matrix $\bld A(k)$ satisfies $\inf_{k>\overline{k}}\min_{i\in\ca N} [A(k)]_{ii} =: \underline{a}>0 $.  \hfill\QEDopen
\end{assumption}
\smallskip 

The concept of NWE in Definition~\ref{def:NetworkEquilibrium} is bound to the particular communication network considered. In the case of a time-varying communication topology, we focus on a different  class of equilibria, namely, those invariant with respect to changes in the communication topology.
\smallskip
\begin{definition}[{Persistent NWE \cite[Ass.~3]{grammatico:18tcns} }]   
A collective vector $\bar{\bld x}$ is a persistent NWE (p-NWE) of \eqref{eq:dynamics_tv_A} if there exists some positive constant $\overline k>0$, such that
\begin{equation}
\label{eq:p-NWE_set}
\bar{\bld x}\in\ca E := \cap_{k>\overline{k}} \:\fix\big(  \boldsymbol{\textrm{prox}}_{\boldsymbol{f}}( \boldsymbol{A}(k)\,\boldsymbol x(k) )   \big)\,.
\end{equation} 
\hfill\QEDopen
\end{definition}  
\smallskip

Next, we assume the existence of a p-NWE.
\smallskip

\begin{assumption}[Existence of a p-NWE]
\label{ass:exist_p-NWE}
The set of p-NWE of \eqref{eq:dynamics_tv_A} is non-empty, i.e., $\ca E\not= \varnothing$.\hfill\QEDopen
\end{assumption}
\smallskip

We note that if the operators $\prox_{\bar f_1}$, $\prox_{\bar f_2}$, $\dots$, $\prox_{\bar f_N}$ have at least one common fixed point, our convergence problem boils down to the setup studied in \cite{fullmer:morse:2018:common_fixed_point_finite_family_paracontraction}. 
 In this case, \cite[Th.~2]{fullmer:morse:2018:common_fixed_point_finite_family_paracontraction} can be applied to  the dynamics in \eqref{eq:dynamics_tv_A} to prove convergence to a p-NWE, $\bar{\bld x} = \bld 1\otimes \overline{x}$, where $\bar{x}$ is a common fixed point of the proximal operators $\prox_{f_i}$'s.
However, if this additional assumption is not met, then convergence is not directly guaranteed. In fact, even though at each time instant the update in \eqref{eq:dynamics_tv_A} describes converging dynamics, the convergence of the overall time-varying dynamics is not proven for every switching signal, see e.g. \cite[Part~II]{Liberzon2003:SwitchingInSystemAndControl}.

\smallskip
Since $\boldsymbol{A}(k)$ satisfies Standing Assumption~\ref{ass:row_stoch} for all $k$, we can apply the same reasoning adopted in the proof of Theorem~\ref{th:convergence_synch_unc} to show that every mapping $\boldsymbol{\textrm{prox}}_{\boldsymbol{f}}\circ \boldsymbol{A}(k)$ is AVG
in a particular space $\ca H_{\bld Q(k)}$, where in general $\bld Q(k)$ is different at different time instants. A priori, there is not a common space in which the AVG propriety holds for all the mappings, thus we cannot infer the convergence of the dynamics in \eqref{eq:dynamics_tv_A} under arbitrary switching signals.  
 In some particular cases, a common space can be found, e.g., if all the adjacency matrices are doubly stochastic, then the time-varying dynamics converge in the space $\ca H_I$, see \cite[Th.~3]{grammatico:18tcns}. 

To address the more general case with row stochastic adjacency matrices, we propose modified dynamics with convergence guarantees to a p-NWE, for every switching sequence, i.e.,
\begin{equation}\label{eq:dynamics_tv_transf}
\boldsymbol x(k+1) = \boldsymbol{\textrm{prox}}_{\boldsymbol{f}}\big( [I + \bld Q(k)(\bld A(k)-I)]\,\boldsymbol x(k) \big) \,.
\end{equation}  

These new dynamics are obtained by replacing $A(k)$ in \eqref{eq:dynamics_tv_A} with $I + \bld Q(k)(\bld A(k)-I)$, where  $\bld Q(k)$ is chosen as in Theorem~\ref{th:convergence_synch_unc}. Remarkably, this key modification makes the resulting operators $\left(\boldsymbol{\textrm{prox}}_{\boldsymbol{f}} \circ \big( I + \bld Q(k)(\bld A(k)-I) \big)\right)_{k \in \mathbb N}$ averaged in the same space, i.e., $\ca H_I$, for all $A(k)$ satisfying Assumption \ref{ass:row_stoch}, as explained in Appendix~\ref{app:tv_theorems}. 
Moreover, the change of dynamics does not lead to extra communications between the agents, since the matrix $\bld Q(k)$ is block diagonal.

\smallskip
The following theorem represents the main result of this section and shows that the modified dynamics in \eqref{eq:dynamics_tv_transf}, subject to arbitrary switching of communication topology, converge to a p-NWE of the game in \eqref{eq:dynamics_tv_A}, for any initial condition.
\smallskip 
 
\begin{theorem}[Convergence of time-varying dynamics]
\label{th:convergence_mod_tv_dyn}
Let Assumptions~\ref{ass:persistent_matrixA}, \ref{ass:exist_p-NWE} hold true. Then, for any $\boldsymbol x(0) \in \bold \Omega $, the sequence $(\bld x(k))_{k\in\bN}$ generated by \eqref{eq:dynamics_tv_transf} converges to a point $\bar{\boldsymbol x}\in \ca E $, with $\ca E$ as in \eqref{eq:p-NWE_set}, namely, a p-NWE of \eqref{eq:dynamics_tv_A}. \hfill\QEDopen 
\end{theorem} 
\begin{proof}
See Appendix~\ref{app:tv_theorems}.
\end{proof}   
\smallskip

%

We clarify that in general, the computation of $Q(k)$ associated to each $A(k)$ requires global information on the communication network. Therefore, this solution is suitable for the case of switching between a finite set of adjacency matrices, for which the associated matrices $Q(k)$ can be computed offline. 

Nevertheless, for some network structures the PF eigenvector is known or it can be explicitly computed locally.
For example, if the matrix $A(k)$ is symmetric, hence doubly-stochastic, or if each agent $i$ knows the weight that its neighbours assign to the information he communicates, the $i$-the component of $q(k)$ can be computed as $\lim_{t\rightarrow \infty} [A(k)^\top]_i^t x = q_i(k)$, for any $x\in\bR^N$ \cite[Prop.~1.68]{bullo:cortes:martinez:2009:distributed_control}. 
Moreover, if each agent $i$ has the same out and in degree, denoted by $d_{i}(k)$, and the weights in the adjacency matrix are chosen as $[A(k)]_{ij} = \frac{1}{d_{i}(k)}$, then the left PF eigenvector is $q(k):=\col( ( d_{i}(k)/\sum_{j=1}^Nd_j(k) )_{i\in\ca N})$.  In other words, in this case each agent must only know its out-degree to compute its component of $q(k)$.  
\smallskip

\section{Proximal dynamics under coupling constraints}\label{sec:coupling_constraint}

\subsection{Problem formulation}
\label{sec:constrained_prox_problem_form}
 
In this section, we consider a more general setup in which the agents in the game, not only are subject to local constraints, but  also to $M$ affine separable coupling constraints. 
Thus, let us consider the \textit{collective feasible decision set}
\begin{equation}
\label{eq:collective_feas_set}
\bld{\ca X} := \bld \Omega \cap \{\bld x\in\bR^{nN}\,|\, C\bld x \leq c\} 
\end{equation}
where $C\in\bR^{M\times nN}$ and $c\in\bR^M$. 
For every agent $i\in\ca N$, the set of points satisfying the coupling constraints reads as
\begin{equation}
\label{eq:X_i_coupling_constr}
\begin{split}
\ca X_i(\bld x_{-i}) := \left\{y\in \bR^n \, | \, C_i y +\textstyle{\sum_{\substack{ j=1\\j\not =i}}^N}C_j x_j \leq c  \right\}\,.
\end{split}
\end{equation}


\begin{stassumption}
\label{ass:convex_constr_set}
For all $i\in\ca N$, the \textit{local feasible decision set} $\Omega_i\cap\ca X_i(\bld x_{-i})$ satisfies Slater's condition. 
\hfill \QEDopen
\end{stassumption}
\smallskip

The original game formulation in \eqref{eq:unc_game_CF} changes to consider also the coupling constraints, i.e.,
\begin{align}
\label{eq:unc_game_CF_CC}
\forall i \in \ca N: \;
\begin{cases}
\textstyle
\argmin_{y \in \mathbb R^n}& \textstyle
 \bar f_i(y)  + \frac{1}{2} \left\| y - \sum_{j=1}^N a_{i,j}x_j \right\|^2\\
\; \text{ s.t. } &  y\in \ca X_i(\bld x_{-i}).
\end{cases}
\end{align}
Hence the correspondent myopic dynamics read as
\begin{equation}
x_i(k+1)= \argmin_{y\in\ca X_i(\bld{x}_{-i}(k))} g_i \left(y,\textstyle\sum_{j=1}^N a_{i,j}x_j(k) \right)\: .
\label{eq:constr_game}
\end{equation}
The concept of NWE (Definition~\ref{def:NetworkEquilibrium}) can be naturally extended for the case of network games with coupling constraints, as formalized next.

\smallskip
\begin{definition}[Generalized Network Equilibrium] \label{def:generalized_NE}
A collective vector $\overline{\boldsymbol x}$ is a  generalized network equilibrium (GNWE) for the game in~\eqref{eq:unc_game_CF_CC} if, for every $ i\in\ca N$,
\begin{equation*}\label{eq:dynamics_constr_1}
\overline{x}_i=  \argmin_{y\in\mathcal{X}_i(\bld x_{-i})} g_i \left(y,\textstyle\sum_{j=1}^Na_{i,j}\overline{x}_j \right) \:.
\end{equation*}
\hfill\QEDopen
\end{definition}
\smallskip


The following example shows that the dynamics in  \eqref{eq:dynamics_constr_1} fail to converge even for simple coupling constraints.

\smallskip
\begin{example}
Consider a 2-player game, defined as in \eqref{eq:unc_game_CF_CC},
where, for $i\in\{1,2\}$,  $x_i \in \bR$ and the local feasible decision set is defined as $ \ca X_i(u) := \{ v \in \bR \, | \, u+v = 0 \} = \{ - u \} $. The game is jointly convex since the collective feasible decision set is  described by the convex set $
\bld{\ca X}:= \{ \bld x \in \bR^2 \, |\, x_1 + x_2 = 0 \}
$. 
The parallel myopic best response dynamics are described in closed form as the discrete-time linear system:
\begin{align}
\begin{bmatrix}
x_1(k+1) \\ x_2(k+1)
\end{bmatrix}
=
\begin{bmatrix}
0 & -1\\
-1 & 0
\end{bmatrix}
\begin{bmatrix}
x_1(k) \\
x_2(k)
\end{bmatrix},
\end{align}
which is not globally convergent, e.g., consider $x_1(0)=x_2(0)=1$. \hfill\QEDopen
\end{example}
\smallskip


The myopic dynamics of the agents fail to converge, thus we  rephrase them  into some analogues pseudo-collaborative ones. In fact, the players aim to minimize their local cost function, while at the same time coordinate with the other agents to satisfy the coupling constraints.
Toward this aim, we first dualize the problem and transform it in an auxiliary (extended) network game  \cite[Ch.~3]{cominetti:facchinei:lasserre}. 
Finally, we design a semi-decentralized iterative algorithm that ensure the convergence of these dynamics to a GNWE.

Let us now introduce the dual variable $\lambda_i\in\bR^M_{\geq 0}$ for each player $i\in\ca N$, and define the concept of \textit{extended network equilibrium} arising from this new problem structure.

\smallskip 
\begin{definition}[Extended Network Equilibrium] \label{def:Extended_NE}
The pair $(\boldsymbol{\overline  x},\bld{\overline{\lambda}} )$,  is an Extended Network Equilibrium  (ENWE) for the game in~\eqref{eq:unc_game_CF_CC} if, for every $ i\in\ca N$, it satisfies
\begin{subequations}
\begin{align}
\label{eq:cond_ENE1}
\overline x_i & = \argmin_{y\in\bR^n} \, g_i \left( y,\textstyle\sum_{j=1}^Na_{i,j}\overline{x}_j \right) + \overline \lambda^\top_i C_iy,\\
\label{eq:cond_ENE2}
\overline \lambda_i & =  \argmin_{\xi\in\bR^M_{\geq 0}} \, -\xi^\top (C\boldsymbol{\overline x}-c) \:.
\end{align}
\end{subequations}
\hfill\QEDopen
\end{definition}
\smallskip

To find an ENWE of the game in \eqref{eq:unc_game_CF_CC}, we assume the presence of a central coordinator. It broadcasts to all agents an auxiliary variable $\sigma$ that each agent $i$ uses to compute its local dual variable $\lambda_i$. Namely, every agent $i$ applies a scaling factor $\alpha_i\in[0,1]$ to $\sigma$ to attain its dual variable, i.e. $\lambda_i=\alpha_i \,\sigma$. The values of the scaling factor represent a split of the burden that the agents experience to satisfy the constraints, hence $\sum_{i=1}^N \alpha_i = 1$. These game equilibrium problems were studied for the first time in the seminal work of Rosen \cite{rosen:65}, where the author introduces the notion of \textit{normalized equilibrium}. We specialize this equilibrium idea for the problem at hand,  introducing the notion of  \textit{normalized extended network equilibrium} (n-ENWE).
\smallskip 
\begin{definition}[normalized-ENWE] \label{def:Extended_NE}
The pair $(\boldsymbol{\overline  x},\overline{\sigma } )$, is a \textit{normalized}-Extended Network Equilibrium (n-ENWE) for the game in~\eqref{eq:unc_game_CF_CC}, if for all $ i\in\ca N$ it satisfies
\begin{subequations}
\begin{align}
\label{eq:cond_nENE1}
\overline x_i & = \argmin_{y\in\bR^n} \, g_i \left(y,\textstyle\sum_{j=1}^Na_{i,j}\overline{x}_j \right) + \alpha_i\,\overline \sigma^\top C_iy,\\
\label{eq:cond_nENE2}
\overline \sigma & = \argmin_{\varsigma\in\bR^M_{\geq 0}} \, -\varsigma^\top (C\boldsymbol{\overline x}-c),
\end{align}
\end{subequations}
where $\alpha_i>0$.
\hfill\QEDopen
\end{definition}
\smallskip

From \eqref{eq:cond_nENE1}~--~\eqref{eq:cond_nENE2}, one can  see that the class of n-ENWE is a particular instance of ENWE and, at the same time, a generalization of the case in which all the agents adopt the same dual variable, hence $\alpha_i=\textstyle{1/M}$, for all $i\in\ca N$. This latter case is widely studied in literature, since describes a fair split between the agents of the burden to satisfy the constraints \cite{yi_pavel:2017:disribiuted_primal_dual_conf,belgioioso_grammatico:2017:semi_decentralized_NE_seeking}. 


Next, we recast the n-ENWE described by the two inclusions (\ref{eq:cond_nENE1})~--~(\ref{eq:cond_nENE2}) as a fixed point of a suitable mappings. 
In fact, \eqref{eq:cond_nENE1} can be equivalently rewritten as $\overline{\bld x} = \boldsymbol {\prox_{f}} (\boldsymbol A \bld x - \bld\Lambda C^\top\overline \sigma )$, where $\bld \Lambda=\diag((\alpha_i)_{i\in\ca N})\otimes I_n$, while \eqref{eq:cond_nENE2} holds true if and only if 
$\overline \sigma = \textrm{proj}_{\bR^M} (\overline{\sigma}+C\boldsymbol{\overline x}-c) $.

To combine in a single operator \eqref{eq:cond_nENE1} and \eqref{eq:cond_nENE2}, we first introduce two mappings, i.e., 
\begin{equation}\label{eq:cal_F_map}
\boldsymbol{\ca R } :=\mathrm{diag} (\boldsymbol {\prox_{f}}, \:\textrm{proj}_{\bR^M_{\geq 0}})
\end{equation} 
and  the affine mapping $\boldsymbol{\ca G}:\bR^{nN+M}\rightarrow \bR^{nN+M}$ defined as
\begin{equation}\label{eq:ca_G_map}
\boldsymbol{\ca G}(\cdot) :=\boldsymbol{G} \,\cdot + \begin{bmatrix}
\boldsymbol 0\\c
\end{bmatrix}  := \begin{bmatrix}
\boldsymbol A & -\bld \Lambda C^\top\\
C & I
\end{bmatrix}\cdot - \begin{bmatrix}
\boldsymbol 0\\c
\end{bmatrix}\:.  
\end{equation}
The composition of these two operators provides a compact definition of n-ENWE for the game in \eqref{eq:unc_game_CF_CC}. In fact, a point $(\boldsymbol{\overline{x}},\overline{\sigma})$ is an n-ENWE if and only if $\col(\boldsymbol{\overline{x}},\overline{\sigma})\in\fix(\bld{\ca R} \circ \bld{\ca G})$, indeed the explicit computation of $\bld{\ca R} \circ \bld{\ca G}$ leads directly to  \eqref{eq:cond_nENE1} and \eqref{eq:cond_nENE2}.

The following lemmas show that an n-ENWE is also a GNWE by proving that a fixed point of $\boldsymbol{\ca R} \circ \boldsymbol{\ca G}$ is a GNWE.
\smallskip

\begin{lemma}[GNWE as fixed point]\label{lemma:GNWE_as_fixed_point} Let Assumption \ref{ass:convex_constr_set} hold true. Then, the following statements are equivalent:
\begin{enumerate}[(i)]
\item $\overline{\boldsymbol x}$ is a  GNWE for the game in \eqref{eq:unc_game_CF_CC};
\item $\exists \overline \sigma\in\bR^{M}$ such that $\col(\overline{\boldsymbol x},\overline \sigma ) \in \fix (\boldsymbol{\ca R} \circ \boldsymbol{\ca G})$.
\hfill\QEDopen
\end{enumerate}
\end{lemma}
\smallskip

\smallskip
\begin{lemma}[Existence of a GNWE]\label{lemma:existence_GNWE}
Let Assumption \ref{ass:convex_constr_set} hold true. It always exists a GNWE for the game in \eqref{eq:unc_game_CF_CC}.
\hfill\QEDopen
\end{lemma}
\smallskip

The two lemmas above are direct consequences of \cite[Lem.~2, Prop.~2]{grammatico:18tcns} respectively, so the proofs are omitted.

\subsection{Algorithm derivation (Prox-GNWE)}
\label{sec:splitting_methods}
In this section, we derive an iterative algorithm seeking n-
GNWE (Prox-GNWE) for the game in \eqref{eq:unc_game_CF_CC}. 
The set of fixed points of the operator $\boldsymbol{\ca R } \circ \boldsymbol{\ca G }$ can be expressed as the set of zeros of other  suitable operators, as formalized next. 

\smallskip 
%
\begin{lemma}[{\cite[Prop.~26.1 (iv)]{Bauschke2010:ConvexOptimization}}] \label{lemma:fixed_point_as_zero}
Let $\bld{\ca B} : = F \times N_{\bR^{M}_{\geq 0}}$, with $F:=\prod_{i=1}^N \partial \bar f_i$. Then,
$\fix \left(\boldsymbol{\ca R } \circ \boldsymbol{\ca G } \right) = \textrm{zer}\left( \bld{\ca B} + \Id - \boldsymbol{\ca G }  \right)$.~\hfill\QEDopen
\end{lemma}
\smallskip

Several algorithms are available in the literature to find a zero of the sum of two monotone operators, the most common being presented in \cite[Ch.~26]{bauschke:combettes}.
We opted for a variant of the preconditioned proximal-point algorithm (PPP) that allows for non self-adjoint preconditioning \cite[Eq.~4.18]{briceno:davis:f_b_half_algorithm}. It results in proximal type update with inertia, that resemble the original myopic dynamics studied. 
The resulting algorithm is here called Prox-GNWE and it is described in  \eqref{eq:Prox-GNWE 1}~--~\eqref{eq:Prox-GNWE 4} where, for all the variables introduced, we adopted the notation $x^+\coloneqq x(k+1)$ and $x = x(k)$, for the sake of compactness. It is composed of three main steps: a local gradient descend \eqref{eq:Prox-GNWE 1}, done by each agent, a gradient ascend \eqref{eq:Prox-GNWE 2}, performed by the coordinator, and finally  an inertia step \eqref{eq:Prox-GNWE 3}~--~\eqref{eq:Prox-GNWE 4}.  

\begin{figure*}[h]
\hrule
\begin{subequations}
\begin{align}\label{eq:Prox-GNWE 1}
\forall i \in \ca N \,:\quad \tilde x_i  &= \prox_{ \textstyle{\frac{\delta_i}{\delta_i+1}} \bar f_{i}} \left(  \textstyle{\frac{\delta_i}{\delta_i+1} } \left(\textstyle{\frac{1}{\delta_i} x_i  + \sum_{j=1}^N} a_{i,j} x_j - \alpha_i C_i^\top \sigma  \right) \right)\\ \label{eq:Prox-GNWE 2}
\tilde \sigma  &= \proj_{\bR_{\geq 0}^M} \left(\sigma +  \textstyle{\frac{1}{\beta}} (C \bld x -c) \right)\\\label{eq:Prox-GNWE 3}
\forall i \in \ca N \,:\quad  x_i^+  &= x_i + \gamma\: \big[  \delta_i(\tilde x_i-x_i) +\textstyle{\sum_{j=1}^Na_{i,j}}(\tilde {x}_j - x_j) - \alpha_i C_i^\top(\tilde \sigma - \sigma) \big]  \\ \label{eq:Prox-GNWE 4}
 \sigma^+ &= \sigma +  \gamma\: [\beta (\tilde \sigma -\sigma ) + C(\bld {\tilde x}-\bld x ) ]
\end{align}
\end{subequations}
\hrule
\end{figure*}

%

The parameters Prox-GNWE are chosen such that the following inequalities are satisfied:
\begin{align}
 r_i-a_{i,i} &< \delta_i^{-1} \leq  \gamma^{-1}- r_i -a_{i,i}, \quad \forall i \in \ca N, \label{cond:coeff_Alg_1} \\
\nonumber
 p_j &< \beta  \leq \gamma^{-1}- p_j,
\end{align}
where 
\begin{equation}\label{eq:radius_def}
\begin{split}
r_i\coloneqq& \textstyle{\frac{1}{2}\sum_{\substack{j=1\\ j\not = i}}^N} (a_{i,j}+a_{ji}) +(1+\alpha_i)\lVert C_i^\top \rVert_\infty,\\
p_j\coloneqq& (1+\alpha_i) \max_{j\in\{1,\dots,N\}}\lVert C_j^\top \rVert_\infty\:.
\end{split}
\end{equation}
We note that, if $\gamma$ is chosen small enough, then the right inequalities in \eqref{cond:coeff_Alg_1} always hold. On the other hand, $\gamma$ is the step size of the algorithm, thus a smaller value can affect the convergence speed. 
%
The formal derivation of these bounds is reported in Appendix~\ref{app:derivation_prox_GNWE}.

We also note that \eqref{eq:Prox-GNWE 1}~--~\eqref{eq:Prox-GNWE 4} require two communication between an agent $i$ and its neighbours one to obtain the value of $x_j(k)$ in \eqref{eq:Prox-GNWE 1} and the second to gather $\tilde x_j(k)$ for all $j\in\ca N_i$ in \eqref{eq:Prox-GNWE 3}.
Finally, we conclude the section by establishing global convergence of the sequence generated by Prox-GNWE to a GNWE of the game in \eqref{eq:unc_game_CF_CC}.

\smallskip
\begin{theorem}[Convergence Prox-GNWE]
\label{th:convergence_Prox_GNWE}
Set $\alpha_i= q_i$, for all $i\in\ca N$, with $q_i$ as in Theorem~\ref{th:convergence_synch_unc}, and choose $\delta_i$, $\beta$, $\gamma$ satisfying~\eqref{cond:coeff_Alg_1}.  Then, for any initial condition, the sequence $(\bld x(k))_{k\in\bN}$ generated by \eqref{eq:Prox-GNWE 1}~--~\eqref{eq:Prox-GNWE 4}, converges to a GNWE of the game in \eqref{eq:unc_game_CF_CC}. \hfill\QEDopen
\end{theorem}
\begin{proof}
See Appendix~\ref{app:proof_conv_prox_GNWE}.
\end{proof}
\smallskip

\section{Numerical simulations}
\label{sec:numerical_simulations}
\subsection{Synchronous/asynchronous Friedkin and Johnsen model}

As mentioned in Section~\ref{sec:problem_formulation}, the problem studied in this paper arises often in the field of opinion dynamics. In this first simulation, we consider the standard Friedkin and Johnsen model, introduced in \cite{friedkin:johnsen:99}. The strategy $x_i(k)$ of each player represents its opinion on $n$ independent topics at time $k$. An opinion is represented with a value between $0$ and $1$, hence $\Omega_i \coloneqq [0,1]^n$. The opinion $[x_i]_j=1$ if agent $i$ completely agrees on topic $j$, and vice versa. Each agent is stubborn with respect to its initial opinion $x_i(0)$ and $\mu_i\in[0,1]$ defines how much its opinion is bound to it, e.g., if $\mu_i=0$ the player is a \textit{follower} on the other hand if $\mu_i=1$ it is \textit{fully stubborn}. In the following, we present the evolution of the synchronous and asynchronous dynamics \eqref{eq:group_dynamics} and \eqref{eq:dynamics_ARock_unc}, respectively, where the cost function is as in \eqref{eq:FJ cost fun}. 

We considered $N=10$ agents discussing on $n=3$ topics. The communication network and the weights that each agent assigns to the neighbours are randomly drawn, with the only constraint of satisfying Assumption~\ref{ass:row_stoch}. The initial condition is also a random vector  $\bld x(0)\in[0,1]^{nN}$. Half of the players are somehow stubborn $\mu = 0.5$ and the remaining are more incline to follow the opinion of the others, i.e., $\mu = 0.1$. For the asynchronous dynamics, we consider three scenarios.
\begin{enumerate}
\item[(A1)] There is no delay in the information and the probability of update is uniform between the players, hence $\overline{\varphi} = 0$  and $p_{\min} = 1/N = 0.1$.
\item[(A2)] There is no delayed information and the update probability is defined accordingly to a random probability vector $p$, so  $\overline{\varphi} = 0$ and $p_{\min} = 0.0191 $. In particular, we consider the case of a very skewed distribution to highlight the differences with respect to (A1) and (A3). 
\item[(A3)] We consider an uniform probability of update and a maximum delay of two time instants, i.e., $p_{\min} = 1/N = 0.1$ and $\overline{\varphi} = 2$. The values of the maximum delay is chosen in order to fulfil condition \eqref{eq:max_delay_bound}.
\item[(A4)] In this case, we consider the same setup as in (A3) but in this case the maximum delay is $\overline{\varphi} = 50$. Notice that the theoretical results do not guarantee the convergence of these dynamics.   
\end{enumerate}
\begin{figure}[ht]
	\centering
	\includegraphics[trim= 30 0 130 40,clip ,width=\columnwidth]{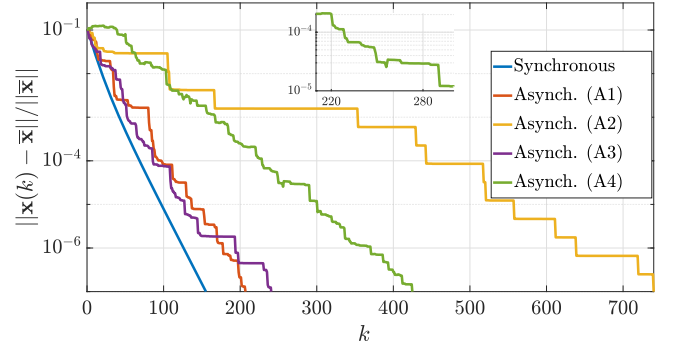}
	\caption{ Comparison between the convergence of the Friedkin and Johnsen model subject to synchronous and asynchronous update. In the latter case the different type of updates (A1), (A2), (A3) and (A4) are considered. For a fair comparison we have assumed that the synchronous dynamics update once every $N$ time instants.}
	\label{fig:FJ_convergence}
\end{figure}
In Figure~\ref{fig:FJ_convergence} it can be seen how the convergence of the synchronous dynamics and the asynchronous ones, in the conditions (A2), (A3), is similar. The delay affects the convergence speed if it surpass the theoretical upper bound in \eqref{eq:max_delay_bound}, as in (A4). It is worthy to notice that, even a big delay, does not seem to lead to instability, while it can produce a sequence that does not converge monotonically to the equilibrium, see the zoom in Fig.~\ref{fig:FJ_convergence}.   
The slowest convergence is obtained in the case of a non uniform probability of update, i.e., (A2). From the simulation it seems that the uniform update probability produces always the fastest convergence.
In fact, the presence of agents updating with a low probability produces some plateau in the convergence of the dynamics.


\subsection{Time-varying DeGroot model with bounded confidence}

In this section, we consider a modified time varying version of the DeGroot model \cite{degroot:1974:reaching_consensus}, where each agent is willing to change its opinion up to a certain extent. This model can describe the presence of extremists in a discussion. 
In the DeGroot model, the cost function of the game in~ \eqref{eq:unc_game_CF} $\bar f_i$, reduces to $\bar f_i = \iota_{\Omega_i}$ for all $i\in \ca N$. We consider $N=8$ agents discussing on one topic, hence the state $x_i$ is a scalar. 
The agents belong to three different categories, that we call: \textit{positive extremists, negative extremists and neutralists}.  
This classification is based on the local feasible decision set. 
\begin{itemize}
\item \textit{Positive (negative) extremists}: the agents agree (disagree) with the topic and are not open to drastically change their opinion, i.e., $\Omega_i = [0.75,1]$ ($\Omega_i = [0,0.25]$).
\item \textit{Neutralists}: the agents do not have a strong opinion on the topic, so their opinions vary from $0$ to $1$, i.e., $\Omega_i=[0,1]$.
\end{itemize}
Our setup is composed of two negative extremists (agent 1 and 2), four neutralists (agents 3-6) and two positive extremists (agents 6 and 8). We assume that, at each time instant, the agents can switch between three different communication network topologies, depicted in Figure~\ref{fig:graph_1}. 
\begin{figure}[ht]
	\centering
	\includegraphics[trim= 20 20 20 20,clip,width=0.3\columnwidth]{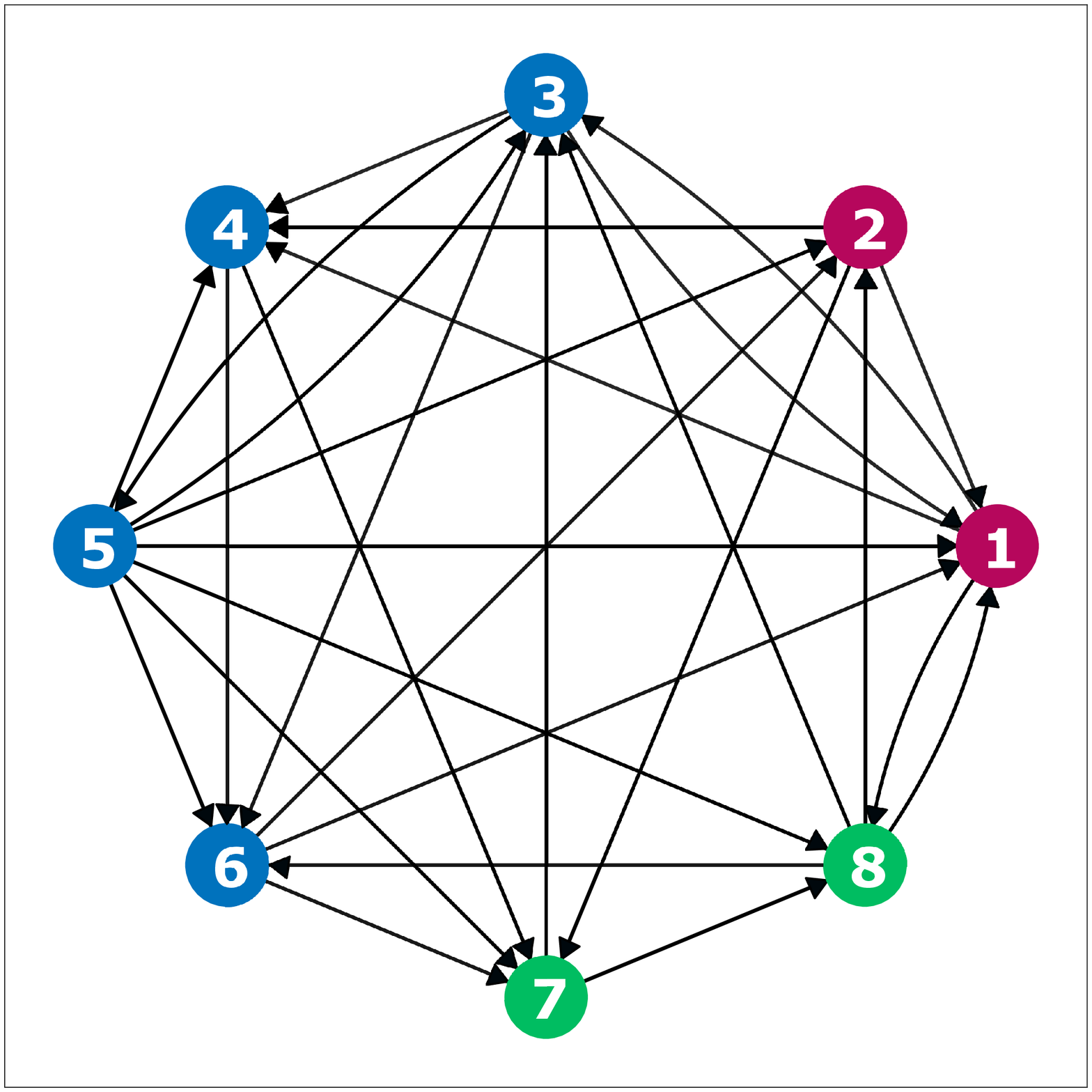}\:
	\includegraphics[trim= 20 20 20 20,clip,width=0.3\columnwidth]{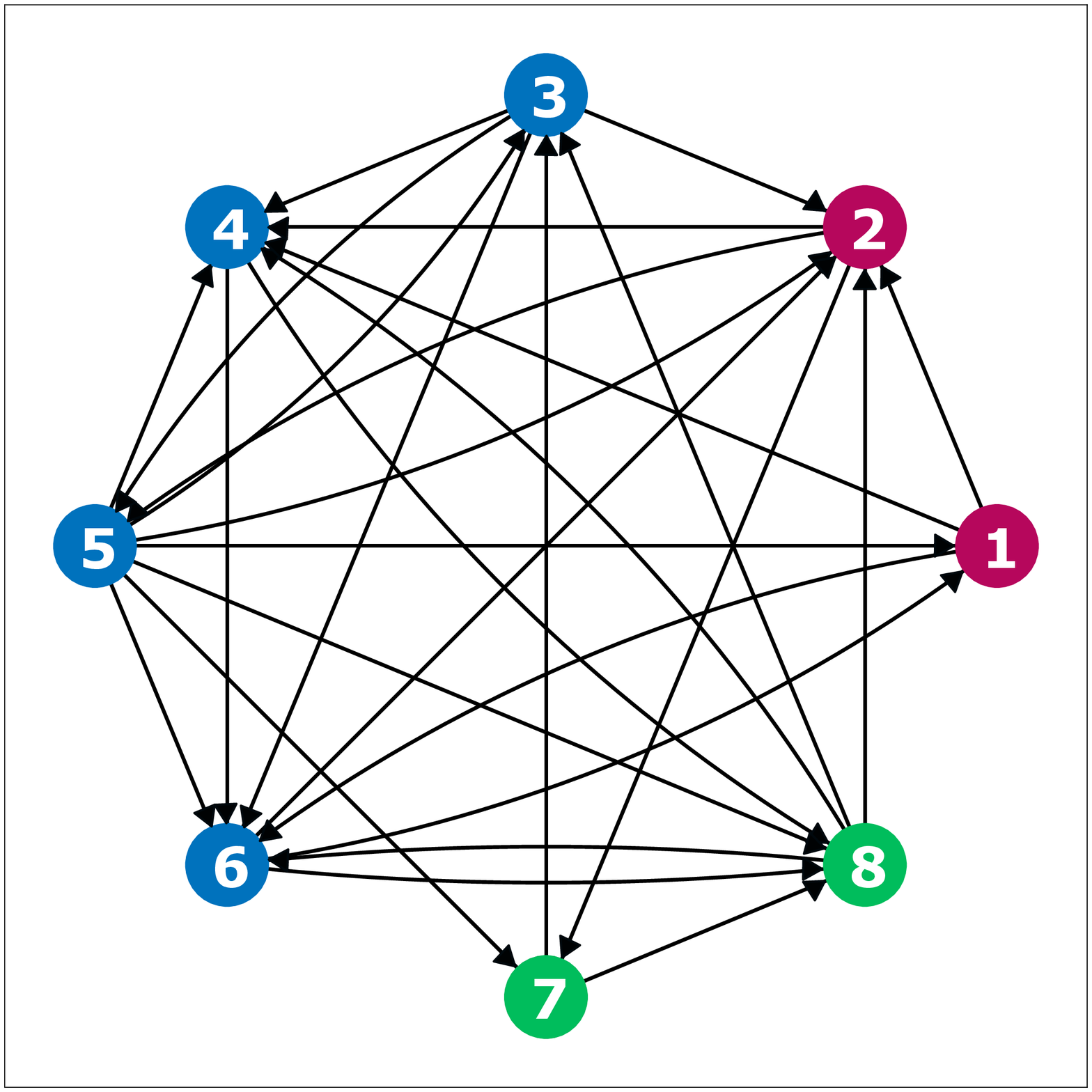}\:
	\includegraphics[trim= 20 20 20 20,clip,width=0.3\columnwidth]{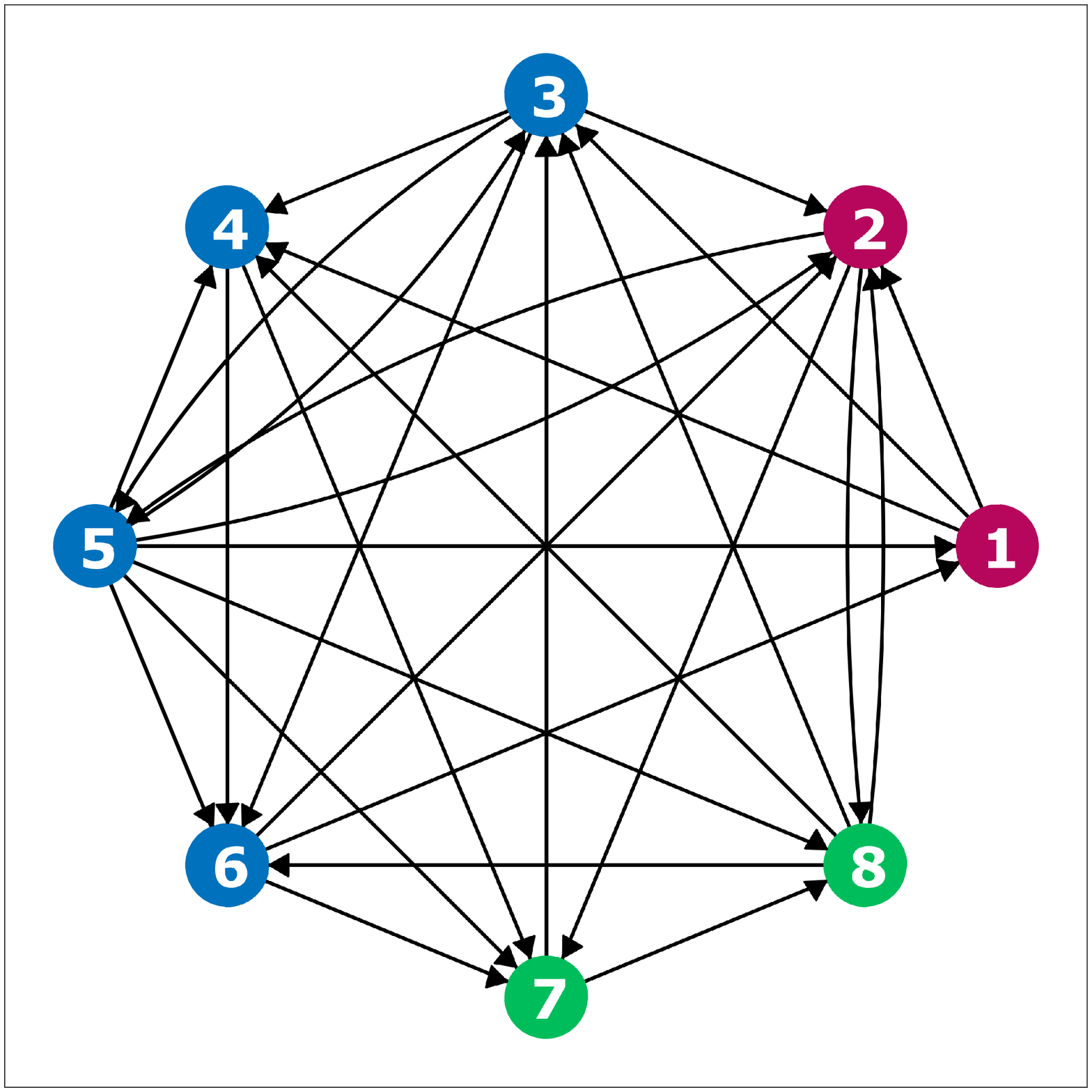}
	\caption{The three different communication networks  between which the agents switch.}
	\label{fig:graph_1}
\end{figure}
The dynamics in \eqref{eq:dynamics_tv_transf}  can be rewritten, for a single agent $i\in\ca N$, as
\begin{equation*}
\label{eq:TV_degroot_model}
x_i(k+1) = \proj_{\Omega_i} \big((1-q_i(k)) x_i(k) + q_i(k)[A(k)]_i \bld x(k) \big).
\end{equation*}
From this formulation, it is clear how the modification of the original dynamics \eqref{eq:dynamics_tv_A} results into  an inertia of the players to change their opinion.

We have proven numerically that Assumption~\ref{ass:exist_p-NWE} is satisfied and that the unique p-NWE of the game is 
$$\bar{\bld x} = \big[0.25,\:0.25,\:0.44,\:0.58,\:0.49,\:0.41,\:0.75,\:0.75 \big]^\top\,.$$ 
In Figure~\ref{fig:FJ_TV_opinion} the evolution players' opinion are reported, as expected, the dynamics converge to the p-NWE, i.e., $\bar{\bld x}$.
\begin{figure}[ht]
	\centering
	\includegraphics[trim=70 0 100 20 ,clip,width=\columnwidth]{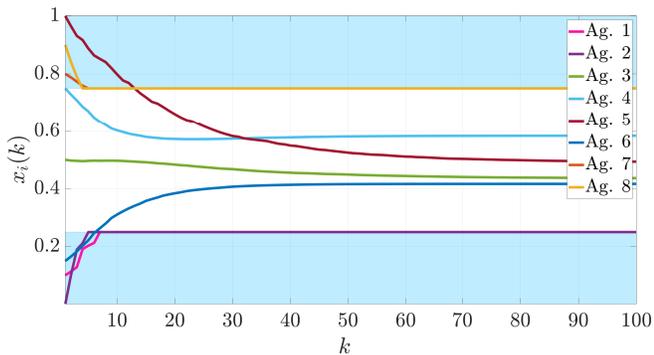}
	\caption{The evolution of the opinion $x_i(k)$ of each agent $i\in\ca N$. The light blue regions represents the local feasible sets of positive and negative extremists.}
	\label{fig:FJ_TV_opinion}
\end{figure}
%
%

\subsection{Distributed LASSO algorithm}
\label{sec:distr_lasso_algorithm}

The following simulation applies Prox-GNWE to solve a problem of distributed model fitting. 
Namely, we develop a distributed implementation of the LASSO algorithm.

First, we introduce the problem by describing the classical centralized formulation, for a complete description see \cite{tibshirani:1996:lasso}. 
Consider a large number of  data $\tilde y\in\bR^d$, where $d=500$, affected by an additive normal Gaussian noise; the true model that generates them is assumed linear and in the form $B x = y	$, where $B\in \bR^{d\times n}$ is the covariate matrix; $x\in\bR^n$ the real parameters of the model, with $n = 6$; and $y$ the model outputs not affected by the noise. Hereafter, the matrix $B$ is assumed to be known. The algorithm aims to compute the parameters estimation $\hat x\in\bR^n$ that minimizes the cost function 
$$f(\hat x)\coloneqq \lVert B \hat  x - \tilde y \rVert^2_2+\lVert \hat x \rVert_1 \, .$$
We focus now on the distributed counterpart of the LASSO, already analized in the literature, e.g.,  see \cite{matteos:distribiuted_linear_regression}. 
In our setup, we assume $N=5$ agents, each one measuring a subset $\tilde y_i\in\bR^{100}$ of $\tilde y$, and communicating via a network described by the adjacency matrix $\bld A$, satisfying Standing Assumption~\ref{ass:row_stoch}. Accordingly, each agent $i$ knows the local covariate matrix $B_i\in\bR^{200\times6}$ associated to the measurements $\tilde y_i\in\bR^{200}$. The noise affecting the data set of each agent $i$, is assumed Gaussian with zero mean and  variance $\sigma_i$. It models the different precision of each agent in measuring the data. The variance $\sigma_i$ is randomly drawn between $[0,\sigma_M]$. In the following, we propose simulations for four different values of $\sigma_M$, namely $0.5\%$, $1.25\%$, $2.5\%$ and $5\%$ of the maximum value of $y$, i.e., the values of $\sigma_M \coloneqq \{1.13 ,\, 2.8  ,\, 5.6  ,\, 11.3\}$. 

We cast the problem as the game in \eqref{eq:unc_game_CF_CC}, where, for all agent $i\in\ca N$, the local cost function is $f_i(\hat x_i)\coloneqq \lVert B_i \hat  x_i - \tilde y_i \rVert^2_2+\lVert \hat  x_i \rVert_1$, where $\hat x_i$ is the local estimation of the model parameters. On he other hand, let $\bld{ \hat x} = \col ((\hat x_i)_{i\in\ca N})$, then the aggregative term $\lVert  \hat  x_i - \bld A \hat{\bld x} \rVert^2$ leads to a better fit of the model, if the weights in the matrix $\bld A$ are chosen inversely proportional to the noise variance of each agent. 
To enforce the consensus between all the local estimations $\hat x_i$, i.e., $\hat x^* \coloneqq \hat x_i(\infty)$ for all $i\in\ca N$, we impose the constraint $\bld L \bld{ \hat x} = 0$, where  $\bld L \coloneqq L\otimes I_n$ and $L$ is the Laplacian matrix associated to the communication network. The resulting game can be solved via Prox-GNWE, that provides a distributed version of the LASSO algorithm. In particular, it generates a sequence $(\bld{ \hat x}(k))_{k\in\bN}$ converging asymptotically to a solution of the model fitting problem.

In Figure~\ref{fig:LASSO_convergence}, it is shown the trajectory of $\frac{\lVert \bld{ \hat x}(k+1) - \bld{ \hat x}(k)\rVert}{\lVert \bld{ \hat x}(k+1)  \rVert}$ where $\bld{ \hat x}(k)$ is the collective vectors of the parameters estimations obtained via Prox-GNWE at iteration $k$; this quantity describes the relative variation of the estimations between two iterations of the algorithms. Notice that the sequence converges asymptotically, but, as expected, a greater noise variance $\sigma_M$ leads to a slower rate of convergence.
\begin{figure}[ht]
	\centering  
	\includegraphics[trim= 0 0 0 0 ,clip,width=\columnwidth]{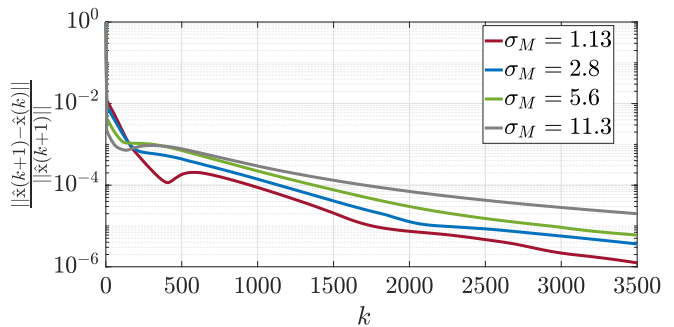}
	\caption{Trajectory of  $\frac{\lVert \bld{ \hat x}(k+1) - \bld{ \hat x}(k)\rVert}{\lVert \bld{ \hat x}(k+1)  \rVert}$ generated via Prox-GNWE. The different maximum noise variances $\sigma_M$ adopted identify the different curves.}
	\label{fig:LASSO_convergence}
\end{figure}
The constraint violation, i.e., $\lVert \bld{L\bld{ \hat x}(k)} \rVert$, at every iteration $k$ of the algorithm are reported in Figure~\ref{fig:LASSO_constraints}; the constraints fulfilment is achieved asymptotically. Also in this case, the use of more noisy measurements $\tilde y$ leads to a slower asymptotic convergence of the local estimations $\hat x_i$ to consensus.
\begin{figure}[ht]
	\centering
	\includegraphics[trim= 0 0 0 0 ,clip,width=\columnwidth]{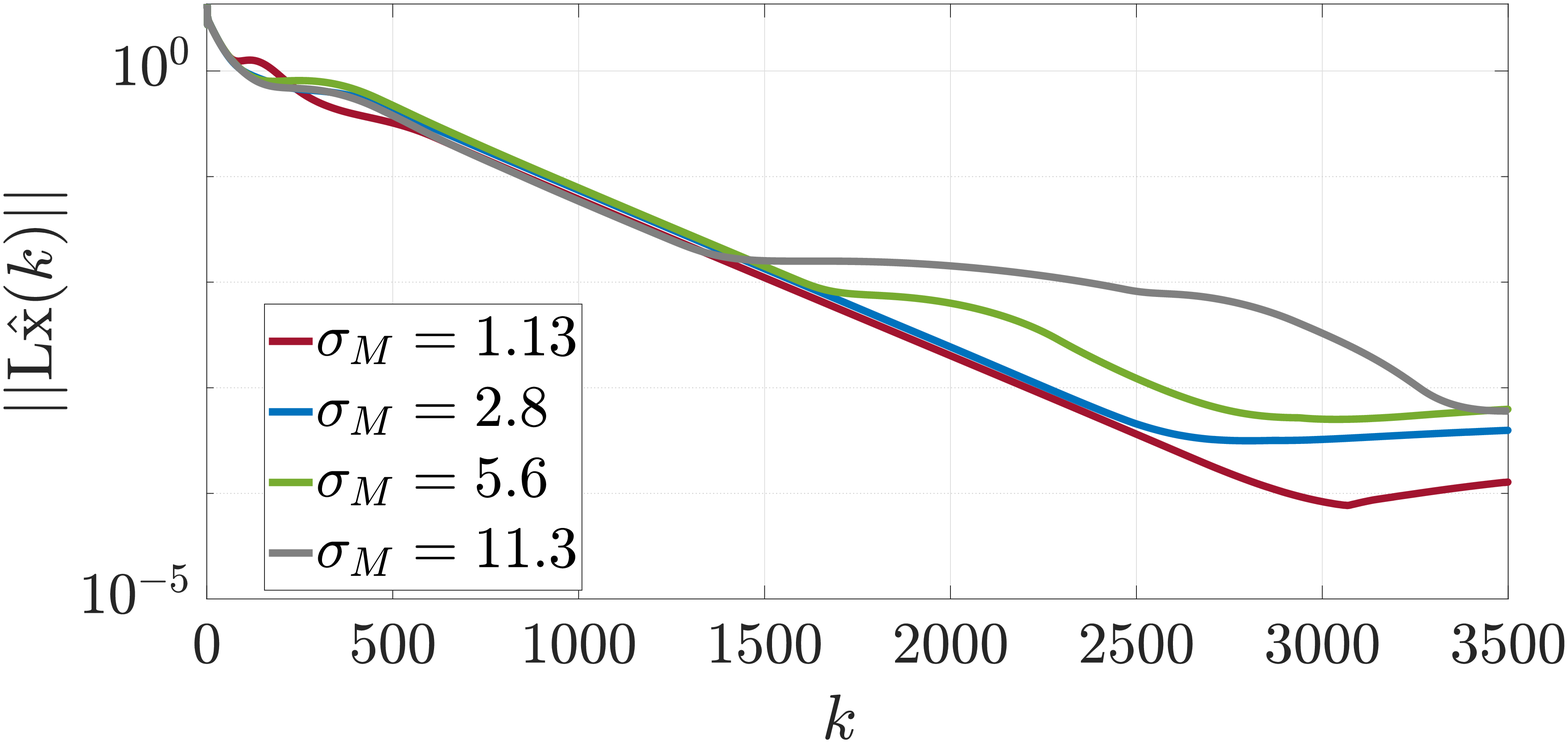}
	\caption{Constraint violation $\lVert \bld{L\bld{ \hat x}(k)} \rVert$ of the collective vector $\hat{\bld x}(k)$ for each iteration $k$. The different maximum noise variances $\sigma_M$ adopted identify the different curves.}
	\label{fig:LASSO_constraints}
\end{figure}

Finally, for every iteration $k$, consider the output of the estimated model $\hat{y}_i(k)=B\hat{ x}_i(k)$, for all $i\in\ca N$; the average of the mean squared values (MSE) of all the local estimations as is computed as $\text{MSE}\coloneqq \frac{1}{N} \sum_{i\in\ca N} \lVert \hat{y}_i(k) - y \rVert^2 $. This value provide an index of the quality of the estimation. In Figure~\ref{fig:LASSO_estimation}, we report the normalized MSE for different values of $\sigma_M$. The increment of $\sigma_M$ causes a slower decrement of the MSE, i.e., a higher number of iterations to achieve a good parameters estimations. On the other hand, if a small noise affect the data, already after $1000$ iterations the parameters obtained perform almost as good as the final ones. The study of this curves emphasises the trade off between a good parameters estimations and a fast converging algorithm.
\begin{figure}[ht]
	\includegraphics[trim= 0 0 0 0,clip,width=\columnwidth]{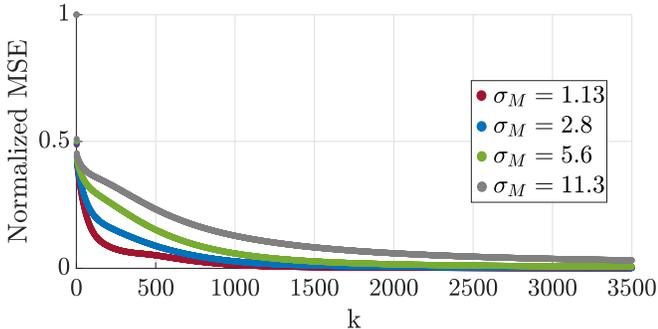}
	\caption{Trajectory of the normalized average  MSE of the estimated model with respect to the real one, for every iteration $k\in\bN$. The different maximum noise variances $\sigma_M$ adopted identify the different curves.}
	\label{fig:LASSO_estimation}
\end{figure}

\section{Conclusion and outlook}
\label{sec:conclusion}
\subsection{Conclusion}
For the class of multi-agent games, proximal type dynamics converge, provided that they are subject only to local constraints and the communication network is strongly connected. We proved that their asynchronous counterparts also converge, and that they are robust to bounded delayed information. If each agent has the possibility to arbitrarily choose the communications weights with its neighbours, then proximal dynamics converge even if the communication network varies over time.  
For multi-agent games subject to both local and coupling constraints, the proximal type dynamics fail to converge, so we provide a converging iterative algorithm (Prox-GNWE) converging to a \textit{normalized-Extended Network Equilibrium} of the game. When the problem at hand can be recast as a proximal type dynamics, the results achieved directly provide the solution of the game, as shown in the numerical examples proposed. 
\subsection{Outlook}
Although we have analyzed several different setups, this topic still presents unsolved issues that can be explored in future research. In the following, we highlight the ones that we think are the most compelling.

The implementation of the asynchronous version of Prox-GNWE presents several technical issues. In fact, the operator in \eqref{eq:PPP_update_rule}, describing the synchronous dynamics, is \textit{quasi}-NE. This is a weaker property than NE, thus it does not allow us to extend the approach used in Section~\ref{sec:asynchrnous_unconstr} to create the asynchronous version of the constrained dynamics, see \cite[Th.~1]{Peng-Yan-Xu:2016:ARock}.
In the literature, there are no asynchronous algorithms providing convergence for \textit{quasi}-NE operators, that translate to implementable iterative algorithms, thus an \textit{ad hoc} solution has to be developed. It is interesting to notice, that in the particular case in which the game is an ordinal potential game, we can easily show that the asynchronous constrained dynamics converge asymptotically.

In the case of a time-varying graph, see Section~\ref{sec:time_var_unconstr}, we have supposed that at every time instant the communication networks satisfies Standing Assumption~\ref{ass:row_stoch}. It can be interesting to weak this assumption by considering jointly connected communication networks, as done in \cite{fullmer:morse:2018:common_fixed_point_finite_family_paracontraction}.

One can notice that, if the adjacency matrix is irreducible, then it is also AVG in some weighted space, see
\cite[Prop.~1]{belgioioso:fabiani:blanchini:grammatico:18css}. The characterization of this space may lead to an extension of the results presented in this paper, as long as the matrix weighting the space preserves the distributed nature of the dynamics.

\appendix
\subsection{Proof of Theorem \ref{th:convergence_synch_unc}}
\label{app:proof_th_1}
Firstly, we introduce the following auxiliary lemma.  

\smallskip
\begin{lemma}
\label{lemma:diagonal_Q_matrix}
Let $P\in\bR^{N\times N}$, $[P]_{ij}\coloneqq  p_{ij}$ be a non negative matrix that satisfies 
\begin{align}
&P \1_N = \1_N, \ \iff \textstyle{\sum_{j\in\ca N}} p_{ij}=1 \label{A1}\\
&q^{\top} P \leq q^{\top}\ \iff \textstyle{\sum_{i\in\ca N}} q_i p_{ij} \leq q_j \label{A2},
\end{align}
where  $q\in\bR^N_{>0}$. Then $P$ is NE in $\ca H_Q$, with $Q=\mathrm{diag}(q)$.   \hfill\QEDopen
\end{lemma}
\smallskip
\begin{proof}
To ease the notation we adopt $\sum_{j} (\cdot) :=\sum_{j\in\ca N} (\cdot)$ and $\sum_{j}\sum_{k<j} (\cdot) :=\sum_{j\in\ca N} \sum_{k\in (1,\dots,j-1)} (\cdot)$ in the following.

Next, we exploit (\ref{A2}) to compute
\begin{align}
\nonumber
x^{\top} P^{\top} Q P x - x^{\top} &P x
= \textstyle{\sum_{i}} q_i \left( \textstyle{\sum_{j}} p_{ij} x_j\right)^2 - \textstyle{\sum_{j}} q_j x_j^2\\\nonumber
 & \leq\textstyle{\sum_{i}} q_i \left( \textstyle{\sum_{j}} p_{ij} x_j\right)^2 - \textstyle{\sum_{j}} \textstyle{\sum_{i}} q_i p_{ij} x_j^2 \\ \label{eq:proof_A_NE_in_H_Q}
&=\textstyle{\sum_{i}} q_i \left(\left( \textstyle{\sum_{j}} p_{ij} x_j\right)^2 - \textstyle{\sum_{j}} p_{ij} x_j^2\right)\:.
\end{align}
Notice that for all $i\in\ca N$ it holds 
$$\textstyle{\sum_{j}} p_{ij} x_j^2=\textstyle{\sum_{j}} p_{ij}(\textstyle{\sum_{k}} p_{ik}) x_j^2\:,$$
furthermore this implies that   
\begin{equation}
\label{eq:proof_A_NE_in_H_Q2}
\left( \textstyle{\sum_{j}} p_{ij} x_j\right)^2 - \textstyle{\sum_{j}} p_{ij} x_j^2 = - \textstyle{\sum_{j}\sum_{k<j} } p_{ij}p_{ik} (x_j - x_k)^2
\end{equation}
Since the matrix $P$ is supposed non negative and from \eqref{eq:proof_A_NE_in_H_Q} and \eqref{eq:proof_A_NE_in_H_Q2}, it follows that 
$$x^{\top} P^{\top} Q P x - x^{\top} Q x \leq  - \textstyle{\sum_{j}\sum_{k<j} } p_{ij}p_{ik} (x_j - x_k)^2\leq 0 $$
Therefore, $P$ is NE in $\ca H_Q$.
\end{proof}

\subsection*{Proof of Lemma \ref{lem:proxA_is_AVG_in_H_Q} }
\noindent
(i) From Standing Assumption~\ref{ass:row_stoch}, $ A$ is marginally stable, with no eigenvalues on the boundary of the unit disk but semi-simple eigenvalues at $1$. It can be rewrite as $ A = (1-\eta)\textrm{Id} + \eta   B$ with $\eta > 1-\underline{a}$, hence $B\geq 0$ and is row stochastic. The graph associated to $B$ has the same edges of the one of $A$, so it is strongly connected and consequently $ B$ is irreducible. The PF theorem for irreducible matrix ensures the existence of a vector $q$ satisfying \eqref{A2} for $B$. Therefore, Lemma~\ref{lemma:diagonal_Q_matrix} can be applied to $B$, implying that $B$ is NE in $\ca H_{ Q}$ where $ Q=\diag(q)\succ 0$.
By construction, this implies that the linear operator $A$ is $\eta$--AVG in the same space. This directly implies that $\bld A=A\otimes I_n$ is $\eta$--AVG in $\ca H_{\bld Q}$, with $\bld Q= Q \otimes I_n$.

\smallskip
(ii) From point (i), we know that $A$ is $\eta$--AVG in the space $\ca H_Q$. 
Define $q_i:=[Q]_{ii}$, since $\prox_{f_i}$ is FNE in $\ca H_{I_n}$ , it holds that for all $ x, y \in\bR^{n}$ and for all $i\in\ca N$, 
\begin{align}\nonumber
\lVert \prox_{\bar f_i}(x&) - \prox_{\bar f_i}(y) \rVert_{I_n} \leq \lVert x - y \rVert_{I_n}  \\ \nonumber
&-\lVert x-y-\prox_{\bar f_i}(x) +\prox_{\bar f_i}(y)  \rVert_{I_n} \\ \label{eq:prox_fne_in_HQ}
\lVert \prox_{\bar f_i}(x&) - \prox_{\bar f_i}(y) \rVert_{q_iI_n} \leq \lVert x - y \rVert_{q_iI_n}  \\ 
&-\lVert x-y- \prox_{\bar f_i}(x) +\prox_{\bar f_i}(y) \rVert_{q_iI_n}\,.\nonumber
\end{align}
Grouping together \eqref{eq:prox_fne_in_HQ}, for every the $i\in\ca N$, leads to 
\begin{equation}
\label{eq:group_prox_fne_in_HQ}
\begin{split}
\lVert \bld{\prox_{f}(x) }- &\bld{\prox_{f}(y)} \rVert_{\bld Q} \leq \lVert \bld x - \bld y \rVert_{\bld Q} - \\
&-\lVert \bld x-\bld y- \bld{\prox_f(x)} +\bld{\prox_f(y)} \rVert_{\bld Q}\,, \\
\end{split}
\end{equation}
for all $\bld{x},\bld y\in\bR^{nN}$, hence $\bld{\prox_f}(\cdot)$ is FNE in $\ca H_{\bld Q}$.
Thus, by \cite[Prop.~2.4]{Yamada2014:Composition_AVG_maps}, the composition $\boldsymbol{\textrm{prox}}_{\boldsymbol{f}} \circ \boldsymbol{A} $ is AVG in $\mathcal{H}_{ \boldsymbol{ Q}  }$ with constant $\textstyle{\frac{1}{2-\eta}}$.

\medskip
Finally, if the matrix $A$ is doubly stochastic, then its left PF eigenvector is $\1$. Applying the result just attained in (ii), we conclude that $\boldsymbol{\textrm{prox}}_{\boldsymbol{f}} \circ \boldsymbol{A}$ is $\eta$--AVG in $\ca H_I$. 
\hfill\QED

\subsection*{Proof of Theorem \ref{th:convergence_synch_unc} }
From Lemma~\ref{lem:proxA_is_AVG_in_H_Q}(ii), we know that the operator $\boldsymbol{\textrm{prox}}_{\boldsymbol{f}} \circ \boldsymbol{A} $ is AVG in $\mathcal{H}_{ \boldsymbol{ Q}  }$ with constant $\textstyle{\frac{1}{2-\eta}}$ and $\eta>1-\underline a$. 
Therefore the convergence proof of the Banach iteration in \eqref{eq:group_dynamics} is completed invoking \cite[Prop.~5.14]{Bauschke2010:ConvexOptimization}.
\hfill $\blacksquare$ 

\subsection*{Proof of Corollary~\ref{cor:krasno_conv}}
The convergence follows directly  from \cite[Prop.~5.15]{Bauschke2010:ConvexOptimization} and the fact that the considered $\bld A$ is nonexpansive in the space $\ca H_{\bld Q}$,  by Lemma~\ref{lemma:diagonal_Q_matrix}.
\hfill\QED

\subsection{Proof of Section~\ref{sec:asynchrnous_unconstr}}
\label{app:proof_th_asynch_un}
Since dynamics \eqref{eq:dynamics_ARock_unc_psi} are a particular case to the ones in \eqref{eq:dynamics_ARock_unc}, we first prove Theorem~\ref{th:convergence_asynch} and successively derive the proof of Theorem~\ref{th:convergence_asynch_psi1} exploiting a similar reasoning.

\subsection*{Proof of Theorem~\ref{th:convergence_asynch}}
From Theorem~\ref{th:convergence_synch_unc}, we know that the operator $T:= \bld{\prox_f\circ A}$ is $\eta$-AVG in the space $\ca H_{\bld Q}$ with $\eta\in (1-\underline{a},1)$. Therefore, it can be written as $T= (1-\eta)\Id + \eta \overline{T}$ where $\overline{T}$ is a suitable NE operator in $\ca H_{\bld Q}$. Notice that  $\fix(\overline{T}) =\fix(T)$. Substituting this formulation of  $T$ in \eqref{eq:dynamics_ARock_unc_psi} leads to 
\begin{equation}
\label{eq:dyn_asynch_proof}
\boldsymbol x(k+1) = \boldsymbol x(k) + \psi_k \eta \zeta_k  \big( \overline{T} - \Id \big) \hat{\boldsymbol x}(k) \,.
\end{equation} 

Since $\overline{T}$ is NE in $\ca H_{\bld Q}$, \cite[Lemma~13~and~14]{Peng-Yan-Xu:2016:ARock} can be applied to the dynamics in \eqref{eq:dyn_asynch_proof}. Therefore, if we chose $\psi_k\in\big( 0,\frac{Np_{\min}}{(2\overline \varphi \sqrt{p_{\min}} +1) (1-\underline{a})}\big)$, the sequence generated by the dynamics in \eqref{eq:dyn_asynch_proof} are bounded and converge almost surely to a point $\overline{\bld x}\in\fix(\overline{T})=\fix(\bld{\prox_f\circ A})$. The proof is completeds recalling that the set of fixed points of $\bld{\prox_f\circ A}$ coincides with the one of NWE, Remark~\ref{rem:NWE_as_fix}.  \hfill\QED

\subsection*{Proof of Theorem~\ref{th:convergence_asynch_psi1}}
The dynamics \eqref{eq:dynamics_ARock_unc} coincide with \eqref{eq:dynamics_ARock_unc_psi}, when $\psi_k=1$ for all $k\in\bN$, therefore if $\frac{Np_{\min}}{(2\overline \varphi \sqrt{p_{\min}} +1) (1-\underline{a})} > 1$, then the convergence is guaranteed from Theorem~\ref{th:convergence_asynch}. From easy computation it can be seen that this condition is equivalent to \eqref{eq:max_delay_bound}, and this conclude the proof. \hfill\QED

\subsection{Proofs of Section~\ref{sec:time_var_unconstr}}
\label{app:tv_theorems}

Before presenting the proof of Theorem~\ref{th:convergence_mod_tv_dyn}, let us introduce two preliminary lemmas. The first proposes a transformation to construct a doubly stochastic matrix from a general row stochastic matrix.
\smallskip
\begin{lemma}
\label{lem:rstoch_in_dstoch}
Let $P\coloneqq p\in \bR^{N\times N}$ and $[P]_{ij}=p_{ij}$ be   a non negative, row stochastic matrix. If there exists a vector $w\in\bR^N_{\geq 0}$ such that $w^\top P = w^\top$, then the matrix 
\begin{equation}
\label{eq:A_transformation}
\overline{P}:= I + \mu \:\mathrm{diag}(w) (P-I) ,
\end{equation}
where $0\leq \mu \leq \frac{1}{\max_{i\in\ca N}{(1-p_{ii})w_i}}$, is non negative and doubly stochastic.
If $\mu < \frac{1}{\max_{i\in\ca N}{(1-p_{ii})w_i}}$, then the diagonal elements of $\overline P$ are positive.
 \hfill\QEDopen
\end{lemma} 
\begin{proof}
 Since $w$ is a left eigenvector of $P$ it holds that
 \begin{equation}
 \label{eq:left_eig_prop}
 w^\top P = w^\top\, .
 \end{equation}
 In order to prove the first part of the lemma we have to show that $\overline P \bld 1 = \bld 1 $ and $\bld 1^\top \overline P = \bld 1^\top $, hence
 \begin{equation}
 \label{eq:sum_row_i_a_bar}
 \overline P \bld 1 = \bld 1-\mu  w + \mu \, \mathrm{diag}(w) P \bld 1   = \bld 1 \, .
 \end{equation} 
Analogously, 
\begin{equation}
 \label{eq:sum_col_i_a_bar}
  \bld 1^\top \overline P = \bld 1^\top-\mu w^\top + \mu w^\top P  = \bld 1^\top \, ,
 \end{equation} 
where the last equality is achieved by means of \eqref{eq:left_eig_prop}. Therefore, from \eqref{eq:sum_row_i_a_bar}--\eqref{eq:sum_col_i_a_bar} we conclude that $\overline P$ is doubly stochastic.

The diagonal elements of $\overline P$ are  $\overline p_{ii}:=1-\mu w_i+\mu w_i p_{ii}$  the off diagonal ones are instead $\overline p_{ij}:=\mu w_i p_{ij}$, $\forall i,j\in \ca N$. From simple calculations, it follows that if $0\leq \mu \leq \frac{1}{\max_{i\in\ca N}{(1-p_{ii})w_i}}$ then $\overline P$ is a non negative matrix.

If $\mu < \frac{1}{\max_{i\in\ca N}{(1-p_{ii})w_i}}$ then $\overline p_{ii}>0$ for all $i\in \ca N$, hence $\overline{P}$ is doubly stochastic and with positive diagonal elements.  
\end{proof}   
\smallskip

%
In the next corollary we consider the case of a matrix $A$ satisfying Standing Assumption~\ref{ass:row_stoch}. 
It shows that the transformation \eqref{eq:A_transformation} does not change the set of fixed points of $A$. Furthermore, it also provides the coefficient of averagedness of $\overline A$.
\smallskip
\begin{lemma}\label{cor:A_transf}
Let the matrix $A$ satisfies Standing Assumption~\ref{ass:row_stoch} and $Q$ as in Theorem~\ref{th:convergence_synch_unc}. Then the matrix
\begin{equation}
\label{eq:A_tran_Q}
\overline{A}:= I +  Q (A-I)
\end{equation}
 is doubly stochastic with self-loops and the following statements hold:
 \begin{enumerate}[(i)]
 \item $\fix(\overline{A}\cdot)=\fix(A\cdot)$;
\item $\overline{A}$ satisfies Standing Assumption~\ref{ass:row_stoch};
 \item $\overline{A}$ is $\chi$-AVG  in $\ca H$ \\ with $\chi\in(\max_{i\in\ca N}{(1-a_{i,i})q_i},1)$.\hfill\QEDopen
 \end{enumerate}
\end{lemma}  
\smallskip  
\begin{proof}
(i) The definition in \eqref{eq:A_tran_Q} is equivalent to  $\overline{A} = (I- Q)+ Q A$. Therefore $\fix(\overline{A})=\fix(A)$, since $Q\succ 0$.
\smallskip

(ii) The matrix $\overline{A}$ is in the form of \eqref{eq:A_transformation} with $\mu = 1$ and, since $1<\frac{1}{\max_{i\in\ca N}{(1-a_{i,i}(k))}}\leq \frac{1}{\max_{i\in\ca N}{(1-a_{i,i}(k))q_i}}$, it satisfies the assumption of Lemma~\ref{lem:rstoch_in_dstoch}, thus $\overline{A}$ is doubly stochastic with self-loops.
Finally, since $A$ satisfies Standing Assumption~\ref{ass:row_stoch}, the graph defined by $\overline A$ is also strongly connected, therefore also $\overline A$ satisfies Standing Assumption~\ref{ass:row_stoch}. 
\smallskip 
 
(iii) To prove that $\overline{A}$ is $\chi$-AVG with $\chi\in(\max_{i\in\ca N}{(1-a_{i,i})q_i},1)$, we can also apply \cite[Lem.~9]{grammatico:18tcns} or use the same argument as in the proof of Lemma~\ref{lem:proxA_is_AVG_in_H_Q}. 
\end{proof}   
\smallskip

\subsection*{Proof of Theorem~\ref{th:convergence_mod_tv_dyn}}
From Corollary~\ref{cor:A_transf}, each matrix $\overline{A} := I + \mu \bld Q(k)(\bld A(k)-I)$ is $(\mu(k)\max_{i\in\ca N}{(1-a_{i,i}(k))q_i(k)})$-AVG  in $\ca H_I$. Since $\bld{ \prox_f}$ is FNE in $\ca H_I$, the operator $\bld{ \prox_f}\circ \overline{\bld A}(k)$ is $\frac{1}{2-\hat\eta(k)}$ where $\hat\eta(k) := \mu(k)\max_{i\in\ca N}{(1-a_{i,i}(k))q_i(k)}$, \cite[Prop.~2.4]{Yamada2014:Composition_AVG_maps}. All the operators used in the dynamics are AVG in the same space $\ca H_I$, therefore the global convergence follows from \cite[Prop.~3.4(iii)]{Yamada2014:Composition_AVG_maps}, since all the cluster point of $(\bld x(k))_{k\in\bN}$ lay in $\ca E$. \hfill\QED

\subsection{Derivation of Prox-GNWE}
\label{app:derivation_prox_GNWE}
Here, we propose the complete derivation of Prox-GNWE.  
To ease the notation, we  define $\bld{\ca A} := \bld{\ca B}+\Id -\bld{\ca G}$, $\bld \varpi:=\col(\bld x(k),\,\sigma (k))$ and $\bld \varpi^+:=\col(\bld x(k+1),\,\sigma (k+1))$. 
Consider the following non symmetric preconditioning matrix
\begin{equation}
\label{eq:preconditioning_matrix}
\Phi := \begin{bmatrix}
\bld{\delta}^{-1}+\bld A & -\bld \Lambda C^\top\\
C & \beta I_M 
\end{bmatrix}
\end{equation}
where $\bld \delta:=\diag((\delta_i)_{i\in\ca N})\otimes I_n$, $\delta_i\in\bR_{>0}$ and $\beta\in\bR_{>0}$. 

%

The  update rule of the modified  PPP algorithm \cite[Eq.~4.18]{briceno:davis:f_b_half_algorithm}, reads as 
\begin{subequations}
\label{eq:PPP_update_rule}
\begin{align}
\label{eq:PPP_update_rule1}
\tilde{\bld \varpi} &=\mathrm{J}_{\Phi^{-1} \bld{\ca A} } \bld \varpi \\
\label{eq:PPP_update_rule2}
\bld \varpi^+ &= \bld \varpi + \gamma \Phi (\tilde{\bld \varpi}-\bld \varpi)
\end{align}  
\end{subequations}
where $\mathrm{J}_{\Phi^{-1} \bld{\ca A} }:=\mathrm{J}_{U^{-1} (\bld{\ca A} +S)}(\Id + U^{-1}S)$ and $\gamma\in>0$.

Let us define its \textit{self-adjoint} and \textit{skew symmetric} components as $U := (\Phi+\Phi^\top)/2$ and $S:=(\Phi-\Phi^\top)/2$. \\
We choose the parameters $\bld \delta$ and $\beta$  to ensure that $U\succ 0 $ and $\lVert U \rVert\leq \gamma^{-1}$, where $\gamma$ will be the step--size of the algorithm. This can be done through the \textit{Gerschgorin Circle Theorem} \cite[Th.~2]{feingold:varga:1962:gerschgorin_circle}. The resulting bounds are reported in \eqref{cond:coeff_Alg_1}.

\smallskip
\begin{remark}
The set of fixed points of the mapping defining the whole update in \eqref{eq:PPP_update_rule1}--\eqref{eq:PPP_update_rule2}, coincides with $\zer(\bld{\ca A})$, \cite[Proof of Th.~4.2]{briceno:davis:f_b_half_algorithm} .\hfill\QEDopen  
\end{remark}
\smallskip

Next, we proceed to compute the explicit formulation of the algorithm. First, we focus on  \eqref{eq:PPP_update_rule1} , so
\begin{subequations}
\begin{align}
\tilde{\bld \varpi} = \mathrm{J}_{U^{-1} (\bld{\ca A} +S)}(\Id + U^{-1}S) \bld \varpi \\
\tilde{\bld \varpi} + U^{-1} (\bld{\ca A} +S)\tilde{\bld \varpi} \ni \bld \varpi + U^{-1}S\bld \varpi\\
\bld 0 \in U(\tilde{\bld \varpi}- \bld \varpi) + \bld{\ca A}\tilde{\bld \varpi} +S(\tilde{\bld \varpi}- \bld \varpi)\\\label{eq:PPP_update_expr}
\bld 0 \in \Phi(\tilde{\bld \varpi}- \bld \varpi) + \bld{\ca A}\tilde{\bld \varpi} 
\end{align}
\end{subequations}  
  
Solving the first row block of \eqref{eq:PPP_update_expr}, i.e. $\bld 0 \in (\bld \delta^{-1} +\bld A)(\tilde{ \bld x} - \bld x) - \bld \Lambda C^\top(\tilde \sigma-\sigma) +F(\tilde{ \bld x})+\tilde{ \bld x} -\bld A \tilde{ \bld x} + \bld \Lambda C^\top \tilde \sigma  $ , leads to
\begin{align*}
\bld 0_{nN} &\in \bld \delta^{-1}(\tilde{ \bld x} - \bld x) -\bld A \bld x + \bld \Lambda C^\top\sigma +F(\tilde{ \bld x})+\tilde{ \bld x}\\
\bld 0_{nN} &\in (\bld \delta^{-1}+I)\tilde{ \bld x} +F(\tilde{ \bld x}) - \bld \delta^{-1}\bld x -\bld A \bld x + \bld \Lambda C^\top\sigma 
\end{align*} 
with a slight abuse of notation let us define $\frac{1}{\bld \delta^{-1}+1}\coloneqq \diag\left(\big(\frac{1}{\bld \delta_i^{-1}+1}\big)_{i\in\ca N}\right)\otimes I_n$
then we obtain
\begin{align}\nonumber
\bld 0_{nN} &\in \tilde{ \bld x} +\textstyle{\frac{1}{\bld \delta^{-1}+1}} F(\tilde{ \bld x}) +\textstyle{\frac{1}{\bld \delta^{-1}+1}}\big[ \bld \Lambda C^\top\sigma - \bld \delta^{-1}\bld x -\bld A \bld x  \big]\\ \label{eq:first_row_block}
 \tilde{ \bld x} &=  \mathrm{J}_{\scriptstyle{\frac{1}{\bld \delta^{-1}+1}} F} \left( \textstyle{\frac{1}{\bld \delta^{-1}+1}}\big[ \bld \delta^{-1}\bld x +\bld A \bld x  -\bld \Lambda C^\top\sigma \big]  \right)
\end{align}

The second row block instead reads as $\bld 0 \in C (\tilde{ \bld x} - \bld x) +\beta (\tilde \sigma-\sigma) +N_{\bR^{M}_{\geq 0}}(\tilde \sigma) + \tilde \sigma - C \tilde{ \bld x} +c $, and leads to 
\begin{subequations}
\begin{align}\label{eq:second_row_block 1}
\bld 0_{M} &\in -C\bld x +\beta (\tilde \sigma-\sigma) +N_{\bR^{M}_{\geq 0}}(\tilde \sigma) +c  \\\label{eq:second_row_block 2}
\tilde \sigma &= \mathrm{J}_{N_{\bR^{M}_{\geq 0}}}\left( \sigma+\textstyle{ \frac{1}{\beta} } (C\bld x -c)  \right)
\end{align}
\end{subequations}

The final formulation of Prox-GNWE, presented in \eqref{eq:Prox-GNWE 1}~--~\eqref{eq:Prox-GNWE 4}, is obtained combining \eqref{eq:first_row_block} and \eqref{eq:second_row_block 2}  with \eqref{eq:PPP_update_rule2}.

\subsection{Convergence proof of Prox-GNWE}
\label{app:proof_conv_prox_GNWE}

Also in this case, we first propose an auxiliary lemma.

\smallskip
\begin{lemma} \label{lemma:max_mon_of_G}
Let $\bld \Lambda = \bld Q$ with $\boldsymbol Q$ is chosen in accordance with Theorem~\ref{th:convergence_synch_unc}. Then the mapping $\Id - \boldsymbol{\ca G}$ from (\ref{eq:ca_G_map}) is maximally monotone in $\mathcal{H}_{\overline {\bld Q} }$ where $\boldsymbol{\overline {\bld Q} } = \mathrm{diag}(\boldsymbol Q,\,I_M)$.
\hfill \QEDopen
\end{lemma}

\smallskip
\begin{proof}
The mapping $\Id - \boldsymbol{\ca G}$ is monotone if and only if $2\boldsymbol{\overline Q}-(\boldsymbol{G}^\top \boldsymbol{\overline Q} + \boldsymbol{\overline Q} \boldsymbol{G})\succeq 0$ (by \cite[Lemma~3]{grammatico:parise:colombino:lygeros:16}).
\begin{equation}\label{eq:G_monotone}
\begin{split}
2\boldsymbol{\overline Q}-(\boldsymbol{G}^\top \boldsymbol{\overline Q} + \boldsymbol{\overline Q} \boldsymbol{G}) & \succeq 0\\
2\boldsymbol{\overline Q} - \left ( \begin{bmatrix}
\boldsymbol{A}^\top \boldsymbol{Q}  &  \boldsymbol{Q} C^\top \\
-C\bld Q & I
\end{bmatrix} + \begin{bmatrix}
\boldsymbol{Q}\boldsymbol{A}   &  -\bld QC^\top \\
\boldsymbol{Q} C & I
\end{bmatrix}\right) & \succeq 0\\
\begin{bmatrix}
2\boldsymbol{Q} - \left ( \boldsymbol{A}^\top \boldsymbol{Q} +\boldsymbol{Q} \boldsymbol{A} \right) &  0\\
0 & 0
\end{bmatrix}  & \succeq 0
\end{split}
\end{equation}
The last inequality in (\ref{eq:G_monotone}) holds if and only if $\Id-\boldsymbol{A}$ is monotone in $\ca H_{\boldsymbol Q}$. The matrix $\bld Q$  is chosen as in Theorem~\ref{th:convergence_synch_unc}, thus $\boldsymbol{A}$ is $\eta$-AVG in $\ca H_{\boldsymbol Q}$ with $\eta\in(0,1-\underline{a})$. From \cite[Example~20.7]{Bauschke2010:ConvexOptimization} we conclude that $\Id-\boldsymbol{A}$ is monotone.  Therefore, also $\Id-\bld{\ca G}$ is monotone. Finally, since we considered a bounded linear operator we conclude the maximally monotonicity of $\Id-\boldsymbol{\ca G}$ in $\ca H_{\overline{\boldsymbol Q}}$ invoking \cite[Example~20.34]{Bauschke2010:ConvexOptimization}.
\end{proof}
\smallskip

\subsection*{Proof of Theorem~\ref{th:convergence_Prox_GNWE}}
By Lemma~\ref{lemma:max_mon_of_G}, we know that $\Id-\bld{\ca G}$ is maximally monotone. Using a similar reasoning to the one in the proof if Theorem~\ref{th:convergence_synch_unc}, we conclude that $\bld{\ca R}$ is firmly nonexpansive in $\ca H_{\overline{\bld{ Q}} }$. The domain of the operator $\bld{\ca R}$ is the entire space, thus applying \cite[Prop.~23.8(iii)]{bauschke:12} we deduce that $\bld{\ca B}$ is maximally monotone.\\  
The mapping $\bld{\ca B} + \Id -\bld{\ca G}$ is the sum of two maximally monotone operators, so \cite[Corollary~25.5(i)]{Bauschke2010:ConvexOptimization} ensures  the maximal monotonicity of the sum. 
The step sizes $\boldsymbol \delta$ and $\beta$ are chosen as in \eqref{cond:coeff_Alg_1}, hence $U$ is positive definite.

Applying \cite[Th.~4.2]{briceno:davis:f_b_half_algorithm} ensures the convergence of the sequence generated by the update rule \eqref{eq:PPP_update_rule1}~--~\eqref{eq:PPP_update_rule2} to a point in $\zer(\bld{\ca B}+\Id-\bld{\ca G})$.
The proof is concluded  invoking Lemma~\ref{lemma:GNWE_as_fixed_point} and \ref{lemma:fixed_point_as_zero}.
\hfill $\blacksquare$

\bibliographystyle{IEEEtran}
\bibliography{diary_bibliography,librarySG}
\vspace{-0.9cm}
\begin{biography}[{\includegraphics[trim=65 0 65 0,width=1in,height=1.25in,clip,keepaspectratio]{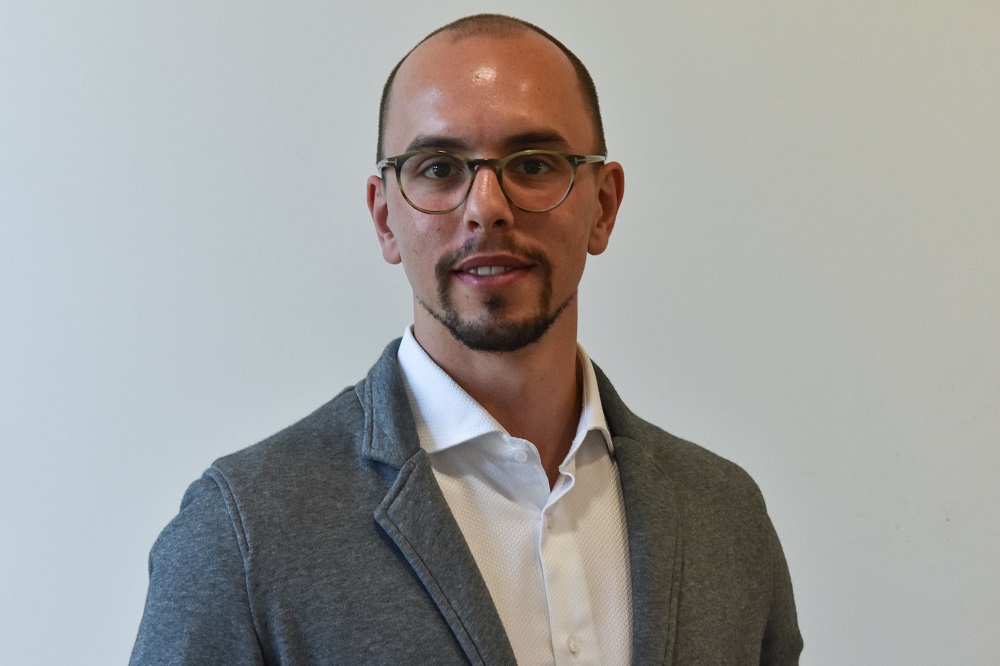}}]{Carlo Cenedese}
Carlo Cenedese is a Doctoral Candidate, since 2017, with the Discrete Technology and Production Automation (DTPA) Group in the Engineering and Technology Institute (ENTEG) at the University of Groningen, the Netherlands. He was born in Treviso, Italy, in 1991, he received the Bachelor degree in Information Engineering in September 2013 and the Master degree in Automation Engineering in July 2016, both at the University of Padova, Italy. From August to December 2016, Carlo worked for the company VI-grade srl in collaboration with the  Automation Engineering group of Padova.
His research interests include game theory, distributed optimization, complex networks and multi-agent network systems associated to decision-making processes.
\end{biography} 
\begin{biography}[{\includegraphics[width=1in,height=1.25in,clip,keepaspectratio]{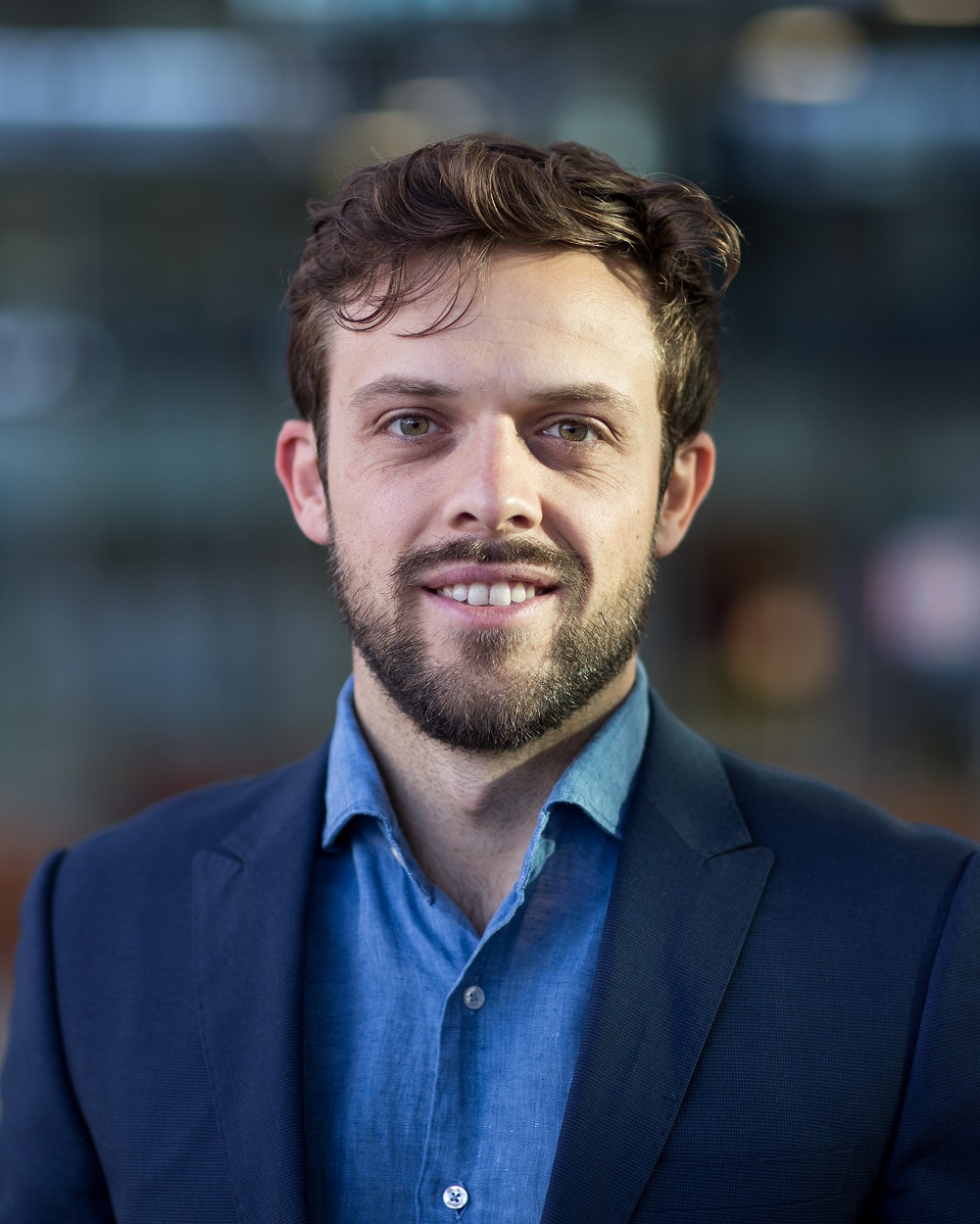}}]{Giuseppe Belgioioso}
Giuseppe Belgioioso is currently a Doctoral Candidate
in the Control System (CS) Group at Eindhoven
University of Technology, The Netherlands. Born in
Venice, Italy, in 1990, he received the Bachelor degree
in Information Engineering in September 2012
and the Master degree (cum laude) in Automation
Engineering in April 2015, both at the University of
Padova, Italy. From Febraury to July 2014 he visited
the department of Telecommunication Engineering
at the University of Las Palmas the Gran Canaria
(ULPGC), Spain. In June 2015 Giuseppe won a
scholarship from the University of Padova and joined the Automation Engineering
group until April 2016 as a Research Fellow. From Febraury to June
2019 he visited the School of Electrical, Computer and Energy Engineering
at Arizona State University (ASU), USA. His research interests include game
theory, operator theory and distributed optimization for studying and solving
coordination, decision and control problems that arise in complex network
systems, such as power grids, communication and transportation networks.
\end{biography} 
\begin{biography}[{\includegraphics[width=1in,height=1.25in,clip,keepaspectratio]{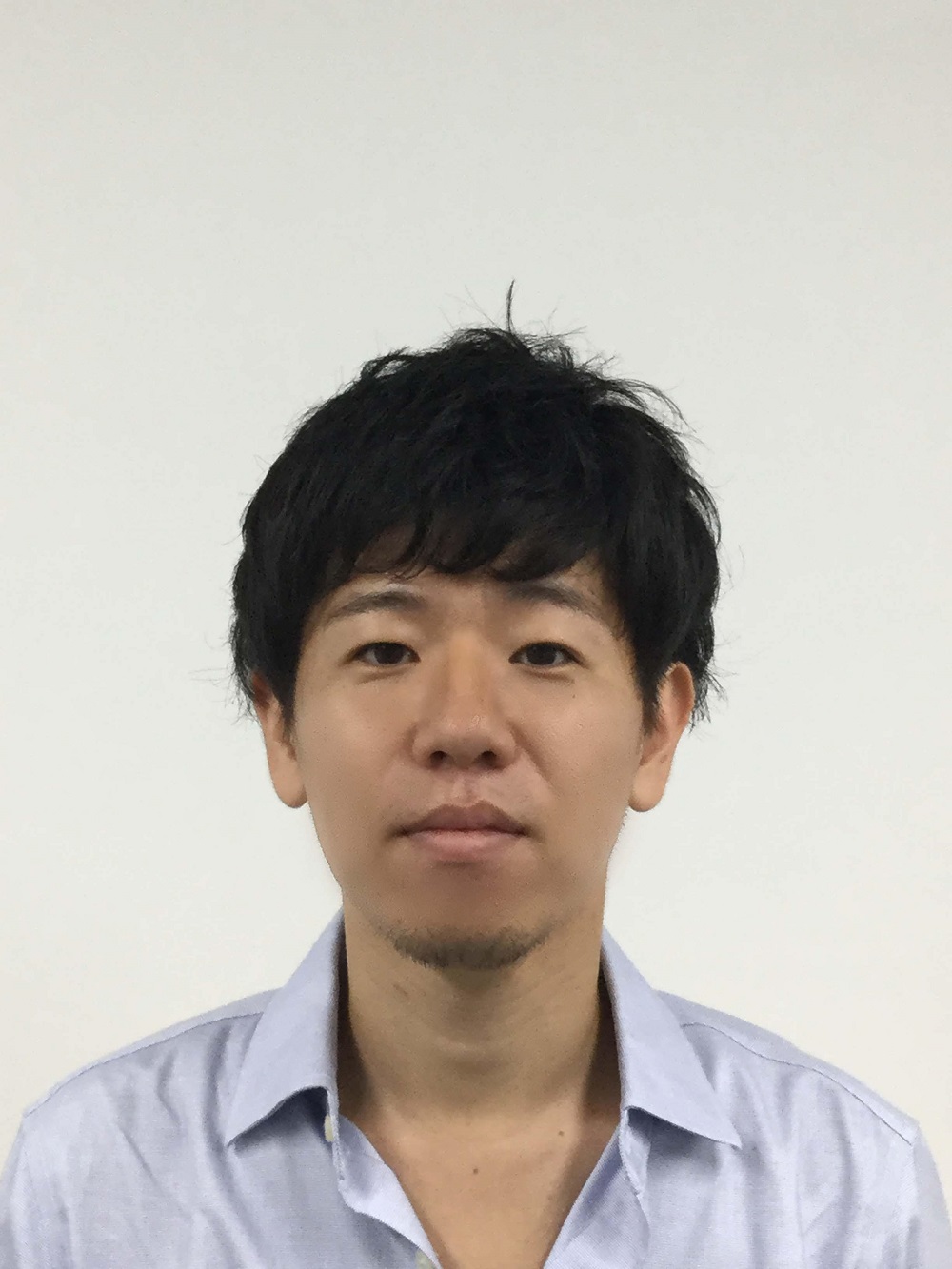}}]{Yu Kawano}(M13) 
is currently Associate Professor in Department of Mechanical Systems Engineering at Hiroshima University. He received the M.S. and Ph.D. degrees in engineering from Osaka University, Japan, in 2011 and 2013, respectively. From October 2013 to November 2016, he was a Post-Doctoral researcher at both Kyoto University and JST CREST, Japan. From November 2016 to March 2019, he was a Post-Doctoral researcher at the University of Groningen, The Netherlands. He has held visiting research positions at Tallinn University of Technology, Estonia and the University of Groningen and served as a Research Fellow of the Japan Society for the Promotion Science. His research interests include nonlinear systems, complex networks, and model reduction.
\end{biography}
\begin{biography}[{\includegraphics[width=1in,height=1.25in,clip,keepaspectratio]{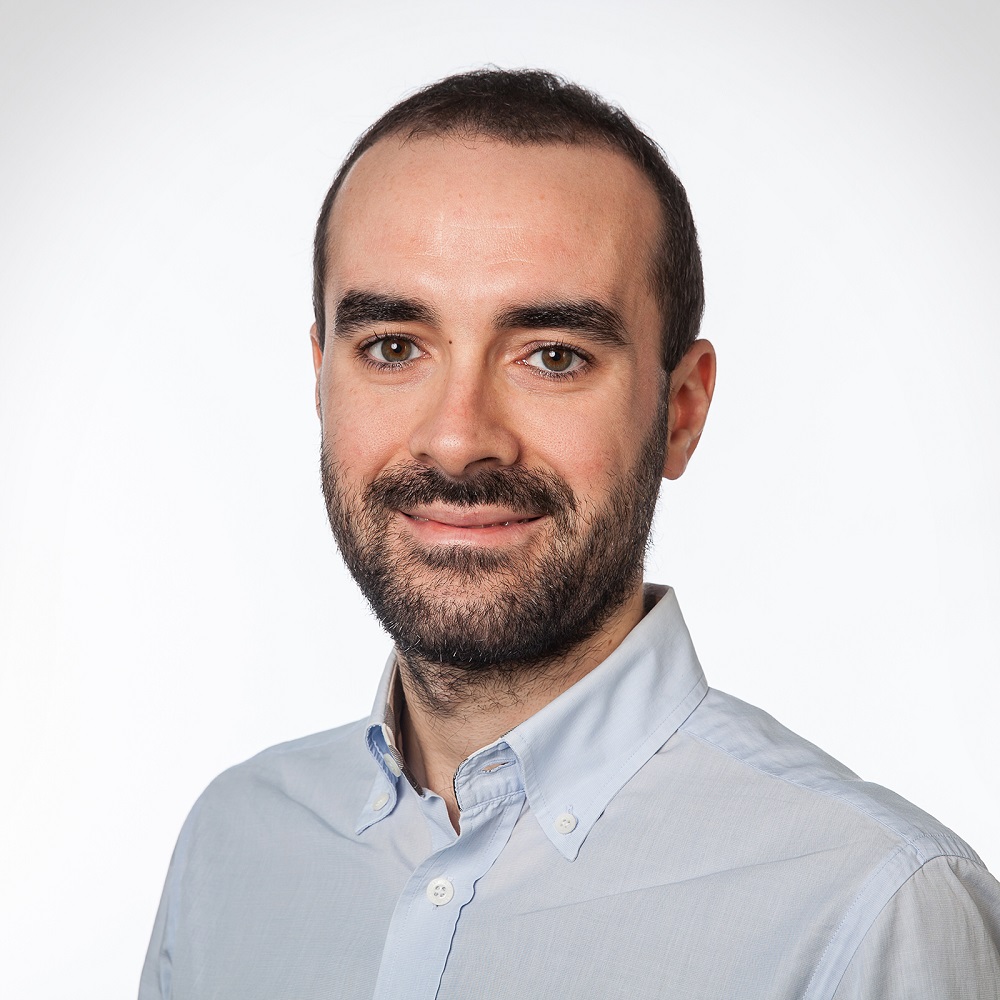}}]{Sergio Grammatico}
(M16,S19) is an Associate Professor in the Delft Center for Systems and Control, TU Delft, The Netherlands. Born in Italy, in 1987, he received the Bachelor degree in Computer Engineering, the Master degree in Automatic Control Engineering, and the Ph.D. degree in Automatic Control, all from the University of Pisa, in 2008, 2009, and 2013 respectively. He also received a Master degree in Engineering Science from the Sant’Anna School of Advanced Studies, Pisa, the Italian Superior Graduate School (Grand École) for Applied Sciences, in 2011. In 2012, he visited the Department of Electrical and Computer Engineering at the University of California at Santa Barbara; in 2013-2015, he was a post-doctoral Research Fellow in the Automatic Control Laboratory at ETH Zurich; in 2015-2017, he was an Assistant Professor in the Control Systems group at TU Eindhoven. He is an Associate Editor for IFAC Nonlinear Analysis: Hybrid Systems, IEEE Transactions on Automatic Control and IEEE CSS Conference Editorial Board. He was awarded 2013 and 2014 TAC Outstanding Reviewer and he was recipient of the Best Paper Award at the 2016 ISDG International Conference on Network Games, Control and Optimization. His research revolves around game theory, optimization and control for complex multi-agent systems, with applications in power grids and automated driving.
\end{biography} 
\begin{biography}[{\includegraphics[width=1in,height=1.25in,clip,keepaspectratio]{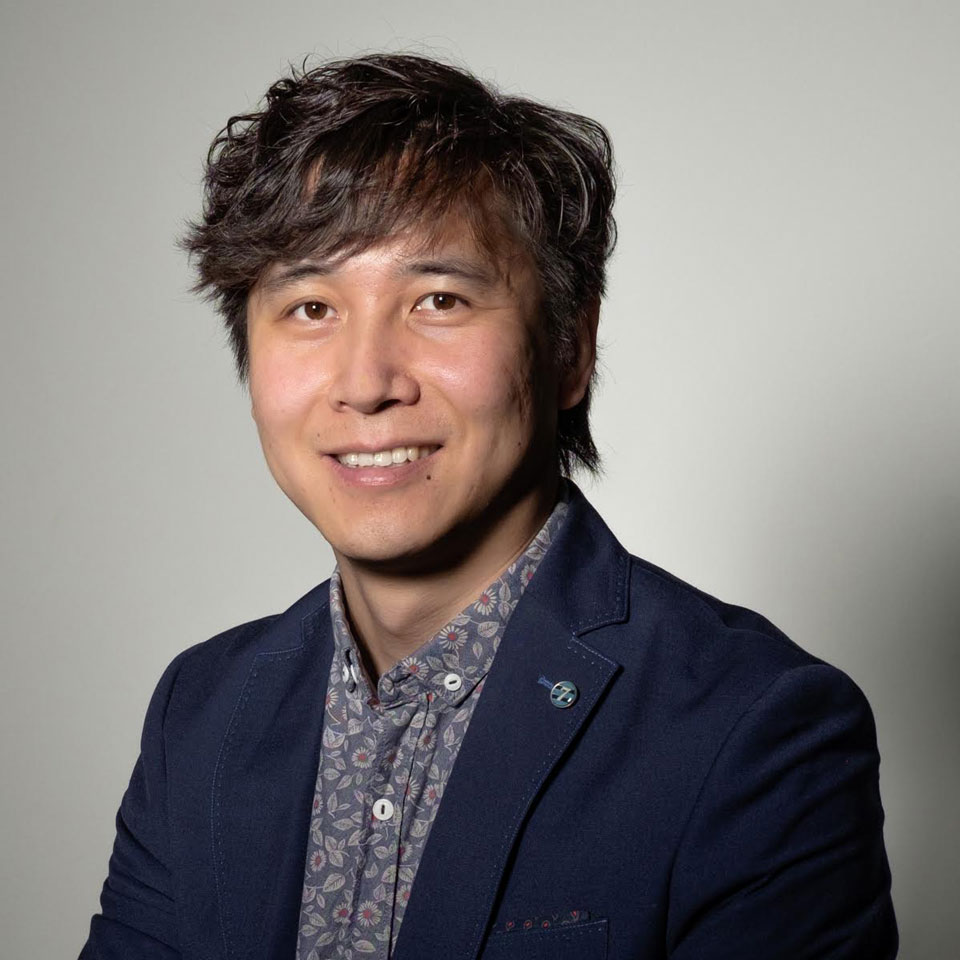}}]{Ming Cao}
has since 2016 been a professor of systems and control with the Engineering and Technology Institute (ENTEG) at the University of Groningen, the Netherlands, where he started as a tenure-track Assistant Professor in 2008. He received the Bachelor degree in 1999 and the Master degree in 2002 from Tsinghua University, Beijing, China, and the Ph.D. degree in 2007 from Yale University, New Haven, CT, USA, all in Electrical Engineering. From September 2007 to August 2008, he was a Postdoctoral Research Associate with the Department of Mechanical and Aerospace Engineering at Princeton University, Princeton, NJ, USA. He worked as a research intern during the summer of 2006 with the Mathematical Sciences Department at the IBM T. J. Watson Research Center, NY, USA. He is the 2017 and inaugural recipient of the Manfred Thoma medal from the International Federation of Automatic Control (IFAC) and the 2016 recipient of the European Control Award sponsored by the European Control Association (EUCA). He is a Senior Editor for Systems and Control Letters, and an Associate Editor for IEEE Transactions on Automatic Control, IEEE Transactions on Circuits and Systems and IEEE Circuits and Systems Magazine. He is a vice chair of the IFAC Technical Committee on Large-Scale Complex Systems. His research interests include autonomous agents and multi-agent systems, complex networks and decision-making processes.  
\end{biography}
 
\end{document}